\documentclass[vi,twoside,reqno]{article}
\pdfoutput=1
\usepackage{ccaption,wrapfig,calc,makeidx,amsfonts, amstext, amscd, amssymb, ifthen, graphicx,amsthm,setspace}
\usepackage[nonamelimits]{amsmath}

\newcommand{\thmformat}{\sf}

\precaption{\setlength{\parskip}{0pt}}

\newif\ifEU
\IfFileExists{euscript.sty}
   {\EUtrue\usepackage[mathcal]{euscript}}
   {}

\captionnamefont{\bf}

\newcommand{\DeltaNotation}[1]{
\ifthenelse{\equal{#1}{notation}}{\vspace{-48pt}}{}
\simplepicholder{1.6in}{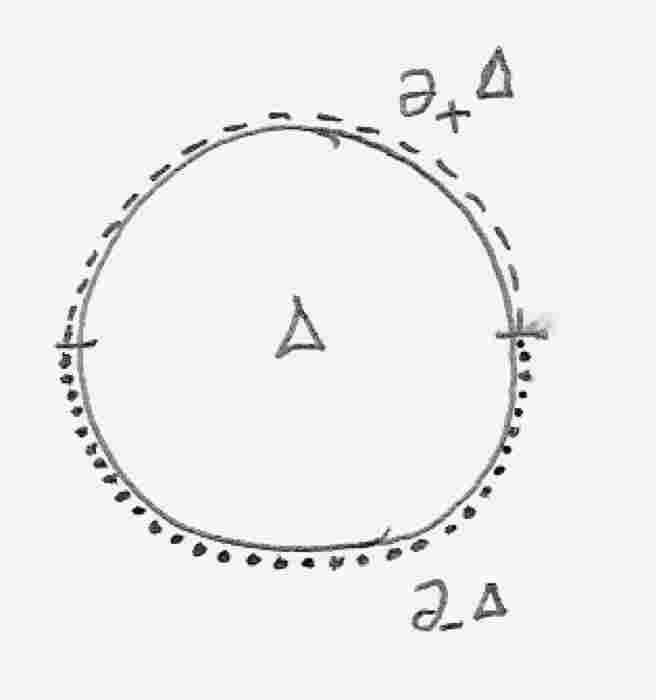}

\begin{tabular}{rl}
	$\partial_+\Delta$: & $y \ge 0$ on $\partial \Delta$ \\
	$\partial_+^o\Delta$: & $y > 0$ on $\partial \Delta$ \\
	$\partial_-\Delta$: & $y \le 0$ on $\partial \Delta$ \\
	$\partial_-^o\Delta$: & $y < 0$ on $\partial \Delta$
\end{tabular}

\caption[Arcs of the unit circle]{Arcs of the unit circle.}
\label{fig:#1:delta_notation}}

\newcommand{\thmLname}[1]{Theorem~\ref{#1}}

\newcommand{\lemLname}[1]{Lemma~\ref{#1}}
\newcommand{\lemMLname}[1]{Lemma~\ref{#1}}

\newcommand{\startevenpage}{%
\ifthenelse{\isodd{\arabic{page}}}{\clearpage}{}}

\newcommand{\figref}[2][]{\textbf{Figure~\mbox{\ref{#2}#1}}}

\newsavebox{\fminibox}
\newlength{\fminilength}
\newenvironment{fminipage}[1]%
  {\setlength{\fminilength}%
        {#1-2\fboxsep-2\fboxrule}%
   \begin{lrbox}{\fminibox}%
   \begin{minipage}{\fminilength}}
  {\end{minipage}\end{lrbox}%
   \noindent\fbox{\usebox{\fminibox}}}

   \newcommand{\Per}{\text{Per}}
   
   \newcommand{\kmax}{\kappa_{\text{max}}}
   
   \newcommand{\fmax}{f_{\text{max}}}
   \newcommand{\CH}{\text{ConvexHull}}

{\end{fminipage}
\end{center}
\end{figure} }

{\end{fminipage}
\end{center}
\end{wrapfigure} }

\newcommand{\titledcaption}[2][]{\caption[#1]{%
\setlength{\parskip}{0pt}%
\ifthenelse{\equal{#1}{}}{}{\textbf{\boldmath#1.} }%
#2}}

\newcommand{\mycaption}[2][]{\caption[#1]{\setlength{\parskip}{0pt}%
#2}}

\newenvironment{dummy}{}{}

\newcommand{\R}{\mathbb{R}}

\newcommand{\sN}{\mathcal{N}}
\newcommand{\sH}{\mathcal{H}}

\newcommand{\sL}{\mathcal{L}}

\newcommand{\sP}{\mathcal{P}}

\newcommand{\partl}[1]{\frac{\partial}{\partial{#1}}}

\newcommand{\operator}[1]{\mathrm{#1}\ }
\newcommand{\operatornosp}[1]{{\mathrm{#1}}}

\newcommand{\Tub}{\operatornosp{Tub}}

\newcommand{\im}{\operator{Im}}
\newcommand{\dom}{\operator{dom}}

\newcommand{\n}{{-1}}
\newcommand{\e}{\epsilon}

\newcommand{\picholder}[2]{\hfill\includegraphics[width=#1]{#2}\hfill{}}
\newcommand{\simplepicholder}[2]{\includegraphics[width=#1]{#2}}

\theoremstyle{plain}

\newtheorem{theoremI}{Theorem}[section]
\newtheorem{lemmaI}[theoremI]{Lemma}
\newtheorem{definitionI}[theoremI]{Definition}
\newtheorem{conjectureI}[theoremI]{Conjecture}
\newtheorem{claimI}[theoremI]{Claim}
\newtheorem{assumptionI}[theoremI]{Assumption}

\newenvironment{definition}{\begin{definitionI}\thmformat}%
{\end{definitionI}}

\newenvironment{claim}{\begin{claimI}\thmformat}%
{\end{claimI}}

\newenvironment{conjecture}{\begin{conjectureI}\thmformat}%
{\end{conjectureI}}

{\end{assumptionI}}

\newenvironment{lemma}{%
\begin{lemmaI}%
\thmformat%
}{\end{lemmaI}}

\newcounter{thm:main} 
\newcounter{cor:main}

\newenvironment{theorem}{%
\begin{theoremI}%
\thmformat%
}{\end{theoremI}}

\newenvironment{theorem*}{\noindent{\bf Theorem.}\em}{\par}
\newenvironment{claim*}{\noindent{\bf Claim.}\em}{\par}
\newenvironment{assumption*}{\noindent{\bf Assumption.}\em}{\par}
\newenvironment{lemma*}{\noindent {\bf Lemma.}\em}{\par}
\newenvironment{conjecture*}{\noindent {\bf Conjecture.}\em}{\par}

\newcommand{\st}{\ |\ }

\newcommand{\half}{\frac{1}{2}} 

\newcommand{\dvol}{\operator{dvol}}
\newcommand{\innp}[2]{\left< #1, #2 \right>}

\newenvironment{mymatrix}{\left(\begin{matrix}}{\end{matrix}\right)}

\begin{document}

\title{Thread-wire surfaces: Near-wire minimizers and topological finiteness}

\author{Benjamin K. Stephens\thanks{The author was partly supported by grant DMS-0244991.  Contact: stephens@math.toronto.edu; www.bkstephens.net.}}
\maketitle

\begin{abstract}
	Alt's thread problem asks for least-area surfaces bounding a fixed ``wire'' curve and a movable ``thread'' curve of length $L$.  We conjecture that if the wire has finitely many maxima of curvature, then its Alt minimizers have finitely many surface components.  We show that this conjecture reduces to controlling near-wire minimizers, and thus begin a three paper series to understand them.  In this paper we show they arise, show that they are embedded, and show that they have a nice parametrization in wire exponential coordinates.  In doing so we prove tools of independent interest: a weighted isoperimetric inequality, a nonconvex enclosure theorem, and a classification of how Alt minimizers intersect planes.  The last item reduces to a question about harmonic functions in the spirit of Rad\'o's lemma.

\textbf{Based on referee comments this article has been split and included in two other papers: ``Existence of thread-wire surfaces, with quantitative estimate'' and ``Near-wire thread-wire minimizers: Lipschitz regularity and localization.''  They are available at http://www.bkstephens.net.  The content included in the new articles is essentially the same, though it is explained better.  Also a constant was changed in the statement of Theorem 1.2 in this paper, and the nonconvex enclosure statement is stated and proved more carefully. -BKS}

\end{abstract}

\newpage

\section{Introduction}

There are two natural ways to state a free boundary problem for the minimal surface equation.  One is the so-called ``thread problem'' $\sP(\Gamma,L)$ which asks for a surface of least area spanning a fixed ``wire'' curve $\Gamma$ in $\R^n$ and a free ``thread curve'' of constrained length $L$.  (This is made precise in Section \ref{sec:def}.)

H.W. Alt posed and showed existence \cite{alt} for a version of the thread problem where each surface is parametrized on disc(s).  Regularity was studied by J.C.C. Nitsche (\cite{nietreg1}, \cite{nietreg2}, \cite{nietreg3}), Alt \cite{alt}, G. Dziuk \cite{dziuk3}, and Dierkes-Hildebrandt-Lewy \cite{dhl}.  Versions of this problem were also studied in Geometric Measure Theory (K. Ecker, \cite{ecker}), and flat chains modulo 2 (R. Pilz \cite{pilz} ).

We study Alt minimizers in $\R^3$, and for smoother wires than Alt studied ($C^4$ instead of rectifiable, and generic in a specific sense).  In a typical minimizer for the Alt problem, the thread curve spends part of its length coinciding with the wire curve, and part of its length free of the wire, bounding surface components.  We call these surface components \emph{crescents} because of their typical shape.  Where the thread is free, its curvature vector has constant length (relating to the Lagrange multiplier for the problem), and lies in the tangent plane of the adjoining crescent.

Minimizers which lie within a small neighborhood of the wire are called \emph{near-wire minimizers.}  They are interesting for three reasons.  First, they are important to settling a finiteness conjecture (Conjecture \ref{conj:finiteness}).  Second they are guaranteed to arise when the thread length is near the wire length (Theorem \ref{thm:nearwire}).  Thirdly, we will show (in two later papers) that they have $C^1$ regularity up to corner points and the normal vector at the corner is forced by local wire geometry.  ($C^1$ up to corners is more regular than what had been shown in the three approaches; it improves on the $C^0$ regularity known for Alt minimizers).  We revisit our three points below, and thus outline this paper.

\subsection{Finiteness conjecture}
Alt's existence proof for the thread problem yields minimizers that may have, a priori, a countable infinity of crescents.  See \figref{fig:intro:ctrex}.  In experiment \cite{myurl} one sees that small crescents always contain a maxima of wire curvature where they meet the wire.  This motivates the following conjecture.

\begin{conjecture} \label{conj:finiteness} Let $\Gamma$ be a space curve which is $C^4$ and generic in the sense of Section \ref{sec:generic}.  Then any Alt minimizer $M$ for $\Gamma$ has at most $C(\Gamma)$ crescents, where $C(\Gamma)$ only depends on the $C^4$ data of $\Gamma$.
\end{conjecture}

Let $W>0$ be a fixed length.  Consider an Alt minimizer $M$.  Divide the crescents into a set $E_1$ with supporting wire length at most $W$ and a set $E_2$ with supporting wire length more than $W$.  By picking $W$ smaller and smaller relative to the geometry of $\Gamma$, we can by a convex hull result (Theorem \ref{thm:chull}) conclude that the crescents in $E_1$ lies in a $R(W)$ tubular neighborhood of $\Gamma$ with $R(W)$ arbitrarily small.  If we can show small crescents near to $\Gamma$ straddle maxima of wire curvature, then we can get

\begin{wrapfigure}[15]{o}{0pt}
	\picholder{2in}{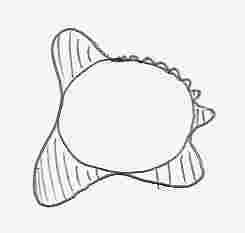}
	\mycaption[Non-generic wire with infinite-component Alt minimizer]{Non-generic wire with infinite-component Alt minimizer.}
	\label{fig:intro:ctrex}	
\end{wrapfigure}

$$|E_1| \le \#\text{ of maxima of }\kappa_\Gamma.$$
On the other hand we always have
$$|E_2| \le \ell(\Gamma)/W.$$
By choosing $W$ correctly relative to the geometry of $\Gamma$, we would prove the conjecture.

\subsection{Near-wire minimizers arise}
It is plausible from both thought experiments and physical ones that if the thread length is near the wire length, any minimizer lies near the wire.  This is nontrivial to prove, however, because it is not a perturbation statement.  When the thread length is near the wire length, it is still quite large, and competitors may range far from the wire.

\begin{theorem}\label{thm:nearwire} Let $\Gamma:[0,\ell(\Gamma)] \to \R^3$ be a $C^3$ non-self-intersecting wire curve parametrized by arclength.
Let $\lambda >0$ be small relative to the $C^3$ data of $\Gamma$.  There is a constant $R(\Gamma,\lambda)>0$ so if $M$ is a minimizer for the thread problem $\sP(\Gamma,\ell(\Gamma)-\lambda)$ 
then $M$ lies in a small radius $R$ normal neighborhood of $\Gamma$:
\begin{eqnarray*}
R(\Gamma,\lambda) &\le& (\pi \kmax/2)^{-1/2}\lambda^{1/2} + o(\lambda^{1/2})\\
\text{Area}(M) &\le& \kmax^\n \lambda + o(\lambda).
\end{eqnarray*}
Here $\kmax = \max |\Gamma''|$ and the error terms do not depend on $M$.
\end{theorem}

We give the proof in Section \ref{sec:nearwire}; it depends on a weighted isoperimetric inequality (Lemma \ref{lem:wtd_iso}) and a nonconvex confinement result for small-area minimal surfaces (Lemma \ref{lem:nearwire}).

\subsection{Nice near-wire parametrization}
A priori, a near-wire minimizer could be a very complicated beast, with crescents intersecting themselves in swirling surfaces which could range far up and down the tubular enclosure, far from their supporting wires.  They could also have branch points.  See \figref{fig:slope:wildbeast}.  The next lemma tames such potential behavior.  It is the first step toward the improved regularity that we show in the subsequent papers.
\begin{figure}
	\picholder{4in}{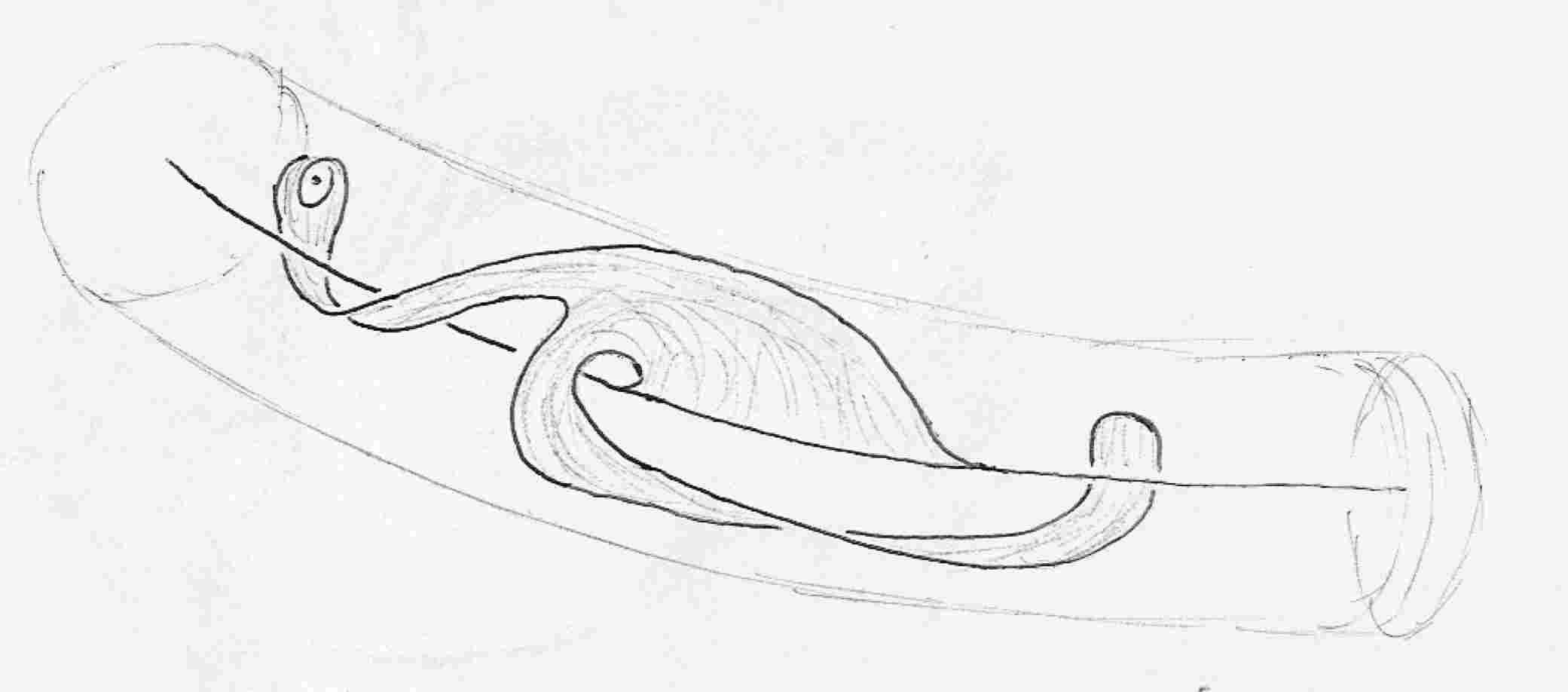}
	\titledcaption[Near-wire Alt minimizers could be quite wild, a priori]{A near-wire Alt minimizer could intersect itself, range far up and down the wire's tubular enclosure, and have branch points.}
	\label{fig:slope:wildbeast}
\end{figure}
\begin{figure}
	\picholder{3.7in}{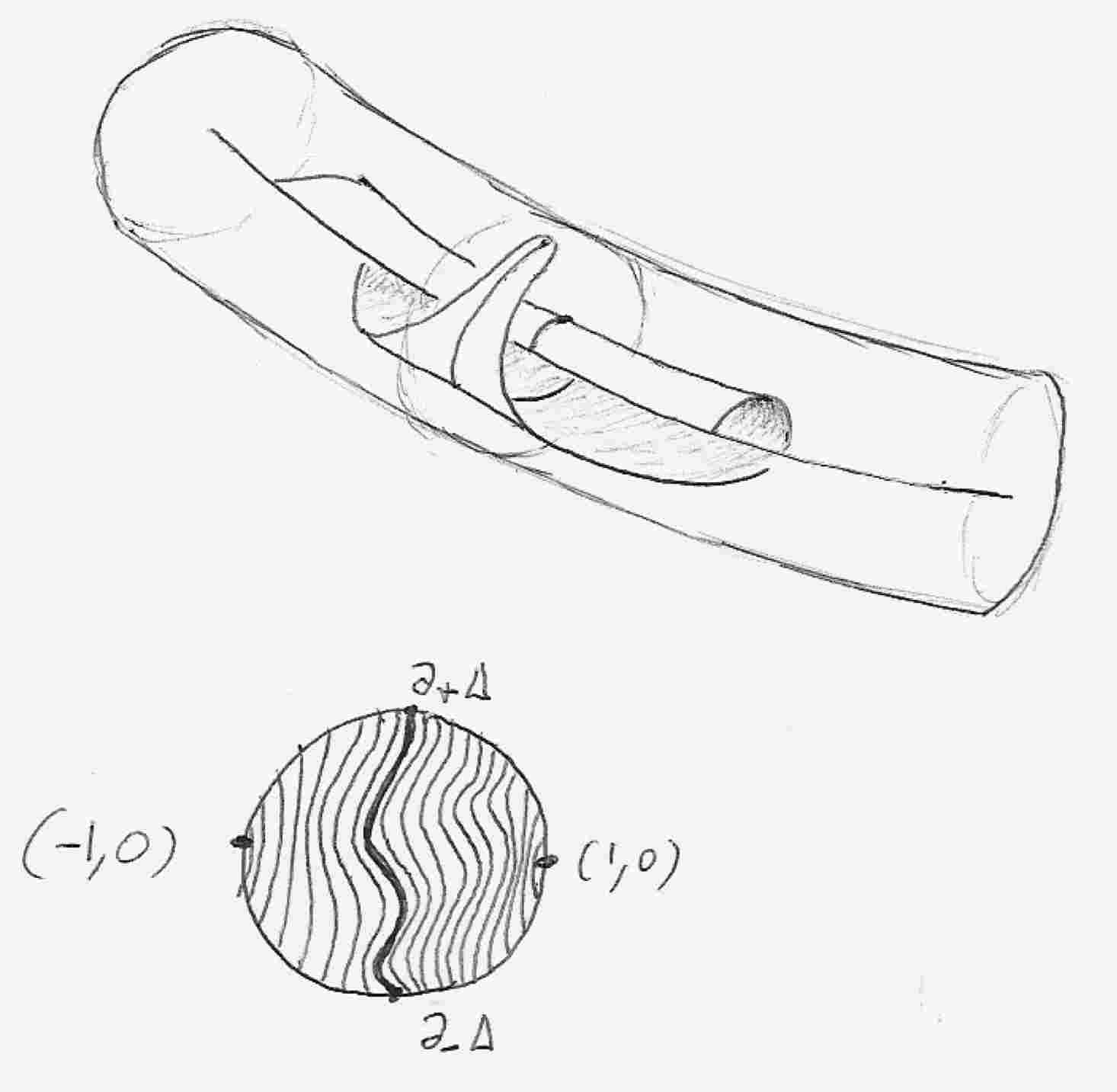}
	\titledcaption{Here we see that the Alt minimizer only lies in slices of the tubular neighborhood which pass through its supporting wire.  Specifically, this crescent is parametrized on the unit disc, with the half-boundaries $\partial_{\pm} \Delta$ mapping to the wire and thread, we may pull back normal discs of the tubular neighborhood to a disjoint family of curves on the disc.}
	\label{fig:slope:slicefriendly}
\end{figure}

\begin{theorem}\label{thm:slope:slicewise}
Let $M$ be a minimizer which is near-wire in the sense that it lies in the largest $R$-tubular neighborhood of $\Gamma$ which does not self-intersect.  Then each crescent of $M$ can be split into continuous curves which correspond bijectively with the slices of the tubular neighborhood corresponding to the supporting wire.  See \figref{fig:slope:slicefriendly}.
\end{theorem}

The main ingredient in this proof is a classification of how Alt minimizers may intersect planes (Section \ref{sec:plane_cresc}).  The result is obtained by studing level sets of harmonic functions (Section \ref{sec:harm_level_set}), in analogy with Classical minimal surfaces and Rad\'o's lemma.

The author would like to thank his thesis advisor, David S. Jerison.
\section{Definitions}\label{sec:def}
\begin{wrapfigure}[20]{r}{1.8in}
\simplepicholder{1.6in}{twsf_delta_disc.jpg}

\begin{tabular}{rl}
	$\partial_+\Delta$: & $y \ge 0$ on $\partial \Delta$ \\
	$\partial_+^o\Delta$: & $y > 0$ on $\partial \Delta$ \\
	$\partial_-\Delta$: & $y \le 0$ on $\partial \Delta$ \\
	$\partial_-^o\Delta$: & $y < 0$ on $\partial \Delta$
\end{tabular}

\caption[Arcs of the unit circle]{Arcs of the unit circle.}
\label{fig:def:delta_notation}
\end{wrapfigure}

Let $\Gamma:[0,\ell(\Gamma)] \to \R^3$ an embedded $C^1$ curve parametrized by arclength and let $L$ be a constant with
\begin{equation}\label{eq:goodLrange}
	|\Gamma(1)-\Gamma(0)| < L < \ell(\Gamma).
\end{equation}
Below we define competitors and an objective function for the \emph{length $L$ thread problem for wire $\Gamma$}; we abbreviate this problem $\sP(\Gamma,L)$.  Generally speaking, Alt used the approach of Rad\'o --- minimize Dirichlet energy in order to get area-minimizing surfaces which, in addition, have a nice parametrization.  (See lemma \ref{lem:intro:morrey} below.)

Each Alt competitor will be a surface obtained by attaching discs to non-overlapping intervals on the wire $\Gamma$.  Let $\Delta$ be the closed unit disc and adopt the notation of \figref{fig:def:delta_notation}.  We define a \emph{thread-wire disc} on $\Gamma$ to be a pair $(X,\phi_-)$ consisting of a map $X \in H^1(\Delta^o,\R^3) \cap C^0(\Delta,\R^3)$ and a continuous map $\phi_-:\partial_- \Delta \to [0,\ell(\Gamma)]$ attaching the disc to the wire:
$$X(p) = \Gamma(\phi_-(p)) \qquad \text{for p }\in \partial_- \Delta.$$
Here $H^1(\Delta^o,\R^3)$ means the space of functions $X:\Delta \to \R^3$ with finite Dirichlet energy:
$$
D(X) := \half \int_{\Delta^o} |X_x|^2 + |X_y|^2 \, dx \, dy.
$$
We assume that $\phi_-$ is weakly monotonic in the sense that $\phi_-(e^{i\theta})$ is non-decreasing for $-\pi \le \theta \le 0.$
We say that two thread-wire discs $(X,\phi_-), (Y,\psi_-)$ are \emph{non-overlapping} if $\im \phi_-$ and $\im \psi_-$ have disjoint interiors.

An \emph{Alt competitor} for $\sP(\Gamma,L)$ is a countable collection $M$ of pairwise non-overlapping thread-wire discs on $\Gamma$ satisfying a length condition:
$$\ell(M) := 
\ell(\Gamma) + \sum_{(X,\phi_-) \in M} (\ell(X|_{\partial_+ \Delta})-|\im \phi_-|) \le L.
$$
The objective function that Alt minimizes is the total Dirichlet energy:
$$D(M) := \sum_{(X,\phi_-) \in M} D(X).$$
We may now state Alt's existence result.  (See \cite{alt}, \cite{dhkw2}.)
\begin{theorem}[Alt, 1973] Let $\Gamma$ be a rectifiable non-self intersecting curve.  Let $L$ satisfy \eqref{eq:goodLrange}.  Then there is an Alt competitor $M^*$ for the Alt problem $\sP(\Gamma,L)$ attaining the infinum of the Dirichlet energy $D$ over all Alt competitors for this problem.  Each crescent of $M$ is harmonic and conformal:
\begin{eqnarray*}
	\Delta X &=& 0\\
	\innp{X_x}{X_y} &=& 0\\
	|X_x| &=& |X_y|.
\end{eqnarray*}
for every $(X,\phi_-) \in M$.  Finally, $M$ uses all the thread length permitted:
\begin{equation}
	\ell(M) = L.
\end{equation}
\end{theorem}

Alt's problem is an optimization subject to a constraint; as such any minimizer $M$ has a Lagrange multiplier.  Specifically, there is a $\kappa \ne 0$ so for any crescent $(X,\phi_-) \in M$, if $\gamma(s)$ reparameterizes $X|_{\partial_+\Delta}$ by arclength then
$$\gamma_{ss} = \kappa \nu(s)$$
where $\nu$ is the outer side-normal to $X$ at the thread.  We call $\kappa$ the \emph{free thread curvature}.  The surface is real analytic on the interior of its domain by Classical regularity \cite{dhkw1}.  Work on boundary regularity at the thread both preceeded and followed Alt's existence work.  The strongest result before this paper was:
\begin{theorem} \textbf{[Hildebrandt, Dierkes, Lewy]}
	If $M$ is an Alt minimizer for the thread problem $\sP(\Gamma,L)$ then each crescent $(X,\phi_-) \in M$ has a real-analytic thread curve $X|_{\partial_+^o \Delta}$.  At any point $p \in \partial_+^o\Delta$, the crescent may be extended to a minimal surface $\tilde{X}$ defined on $\Delta \cup B_\e(p)$.  At $p$ there may be a branch point, but only of even order.
\end{theorem}

Finally, we recall that minimizing Dirichlet energy minimizes the area:
\begin{lemma}\label{lem:intro:morrey} \textbf{[Morrey]} Let $X$ be of class $C^0(\Delta,\R^3) \cap H^1(\Delta^o,\R^3)$.  Then, for every $\e>0$, there exists a homeomorphism $\tau_\e$ of $\Delta$ onto itself which is of class $H^1$ on $\Delta$ which reparametrizes $X$ as $Z_\e = X \circ \tau_\e$ so that$$Z_\e \in C^0(\Delta,\R^3) \cap H^1(\Delta^o,\R^3)$$
	and
	$$D(Z_\e) \le A(X) + \e.$$
	Here $A(X)$ is the area.
\end{lemma}
\section{Proof of Theorem \ref{thm:nearwire}}\label{sec:nearwire}

In this section we prove that when the thread length is near the wire length, any minimizer to the thread problem is near-wire.  Let us make the notion of near-wire precise.

\begin{definition}\label{def:normal_nbhd} Let $\Gamma(t)$ be a regular $C^1$ space curve.  Let $D_r(\Gamma,t')$ be the radius-$r$ disc the center of which $\Gamma$ pierces normally at $t=t'$.  We define the \emph{normal $r$-neighborhood of $\Gamma$} to be the union of these normal discs:
	$$N_r (\Gamma) = \cup_{t\in\dom \Gamma} D_r(\Gamma,t).$$
	We say the normal neighborhood is \emph{simple} if the normal discs are pairwise disjoint.  We say an Alt minimizer on $\Gamma$ is a \emph{near-wire minimizer} if it lies in a simple normal neighborhood of $\Gamma$.
\end{definition}

\subsection{Constructing a good near-wire competitor}
We begin by constructing a near-wire Alt competitor with small Dirichlet energy.  Then we show that to best this, any Alt minimizer must itself be near the wire.
If $\Gamma$ is a straight segment, the condition of the theorem is impossible to meet and we are done.  Otherwise, the curvature of $\Gamma$ attains its maximum $\kmax > 0$ at some $s_0$.  Without loss of generality, assume that $\Gamma$ is parametrized proportional to arclength.  Let $G(s) = \Gamma(s_0+s)$.  Then by Taylor's theorem, 
$$G(s) = \left(s - \frac{\kmax}{6} s^3, \frac{\kmax}{2} s^2, \frac{\kmax T_\Gamma(s_0)}{6} s^3 \right) + o(s^3).$$
in Frenet coordinates for $s$ near $0$.  For small $w>0$ we attach a surface component to $\Gamma$ using
\begin{eqnarray*}
	\Xi(x,y) &=& x G'(0) + (y w^2/2 + (1-y)x^2/2)G''(0) + O(x^3).\\
	\Xi_x &=& G'(0) + (1-y)x G''(0) + O(x^2)\\
	\Xi_y &=& (w^2/2 - x^2/2) G''(0) + O(x^2)\\
	|\Xi_x \times \Xi_y| &=& (w^2-x^2)/2 \kappa_G(0) + w^2 O(x^2) + O(x^3)
\end{eqnarray*}
for $(x,y) \in [-w,w] \times [0,1]$.
Then
\begin{eqnarray*}
	A(\Xi)& =& \frac{2}{3} w^3 \kappa_\Gamma(s_0) + O(w^4).\\
	\ell(\Xi|_{[-w,w]\times\{0\}}) &=& 2w \\
	\ell(\Xi|_{[-w,w]\times\{1\}}) &=& \left| \left(2w - \frac{\kmax^2}{3}w^3, 0, -\frac{\kmax T_\Gamma(s_0)}{3} w^3\right) + o(w^3) \right| \\
	&=& 2w - \frac{2 \kmax^2}{3} w^3 + o(w^3)
\end{eqnarray*}
Using Morrey's Lemma (Lemma \ref{lem:intro:morrey}), we can find a thread-wire disc $(X,\phi_-)$ with
\begin{eqnarray*}
	D(X) &=& \frac{2}{3} \kmax w^3+ O(w^4)\\
	\ell(X|_{\partial_+ \Delta}) &=& 2w - \frac{2}{3} \kmax^2 w^3+ o(w^3)\\
\ell(X|_{\partial_-\Delta}) &=& 2w.
\end{eqnarray*}
Pick
\begin{equation}\label{eq:nearwire:w}
	w = \left(\frac{2\kmax^2}{3}\right)^{-1/3} \lambda^{1/3} + o(\lambda^{1/3})
\end{equation}
so that the Alt competitor $P_0 = \{(X,\phi_-)\}$ satisfies
\begin{eqnarray}\label{eq:nearwire:P0}
	D(P_0) &=& \kmax^\n \lambda + o(\lambda)\\
	\nonumber \ell(P_0) &=& \ell(\Gamma) - \lambda
\end{eqnarray}
and is thus admissible for the thread problemn $\sP(\Gamma,L)$ when $L$ satisfies \eqref{eq:goodLrange}.
\newcommand{\Ghat}{\hat{\Gamma}}
\subsection{Minimizer's thread lies near wire}
Now we consider an Alt minimizer $M$ for $\sP(\Gamma,L)$ with $L$ satisfying \eqref{eq:nearwire:Lrange} for small $\lambda$ to be determined later. As a minimizer, it must beat $P_0$, so by \eqref{eq:goodLrange} and \eqref{eq:nearwire:P0}, we know
\begin{equation}\label{eq:nearwire:DM}
 D(M) \le \kmax^\n \lambda + o(\lambda).
\end{equation}

We parameterize space near $\Gamma$ using a parallel orthonormal frame $E_1,E_2$ along $\Gamma$:
\begin{eqnarray*}
 E_i'(s) &=& -\innp{E_i}{\Gamma''(s)} \Gamma'(s), \qquad i=1,2\\
\exp:(s,x,y) &\mapsto &\Gamma(s) + x E_1(s) + y E_2(s).
 \end{eqnarray*}
Relative to the bases $\partl{s},\partl{x},\partl{y}$ and $\Gamma'(s), E_1(s), E_2(s)$ we have
$$d\exp = \begin{mymatrix}
 1 - \innp{x E_1 + yE_2}{\Gamma''(s)}  & 0 & 0\\
 0 & 1 & 0 \\
 0 & 0 & 1
\end{mymatrix}$$
Let $r = \sqrt{x^2+y^2}.$ Let $R_0$ be small enough so the $R_0$ normal neighborhood of $\Gamma$ is simple. Let $N_R$ be the $R$ normal neighborhood of $\Gamma$ for $R < R_0/2$. Then $\exp$ gives a diffeomorphism onto $N_R$ with
\begin{equation}\label{eq:nearwire:len_exp}
 |d\exp v| \le |v|(1 + 4r \kappa_\Gamma(s) + O(r^2)).
\end{equation}
Define a map $\Psi$ from $N_R$ to the strip $T_R = [0,\ell(\Gamma)]\times[0,R]$ by composing $\exp^\n$ with the map
$$(s,x,y) \mapsto (s,r).$$

The map $\Psi$ will allow us to project pieces of crescents of the Alt minimizer $M$ to the narrow rectangle $T_R$. 
There we will estimate areas and lengths to show that it is not profitable to leave the strip. It will be technically easier to work with sets whose boundaries consist of \emph{finitely many} smooth curves.  In the following argument we will approximate objects at several places using small positive constants $\e_1,\e_2,\e_3,\e_4,\e_5,\e_6$.  These constants will be small relative to the $C^3$ data of $\Gamma$ (i.e. its length, and the sup norms of its first three derivatives); at the end we will let them become arbitrarily small and obtain the desired result.  See \textbf{Table 1}.

\begin{table}
	\begin{tabular}{rp{4.2in}}
	qty. & role\\
	$\e_1$ & controls thread-length change from selecting finite subset of surface components of $M$\\
	$\e_2$ & controls length of segments in piecewise linear approximation $\Ghat_{\e_2}$ of $\Gamma$\\
	$\e_3$ & equals radius of a pipe surface about $\Ghat$ containing $\Gamma$; given by Lemma~\ref{lem:pipe} \\
	$\e_4$ & controls length of subcurves which we break thread into \\
	$\e_5$ & controls area error committed in modifying thread of $M$\\
	$\e_6$ & equals radius of tubular nbhd. of $\Gamma$ containing pipe surface $P_{\e_3}(\Ghat_{\e_2})$; given by Lemma~\ref{lem:pipe}.
	\end{tabular}
	\caption{Roles of $\e$'s.}
\end{table}

Let $M'\subset M$ be an Alt competitor consisting of finitely many crescents of $M$ such that
\begin{equation}\label{eq:Mprimelen}
	\ell(M')\ge\ell(M)-\e_1.
\end{equation}
\begin{figure}
\picholder{4in}{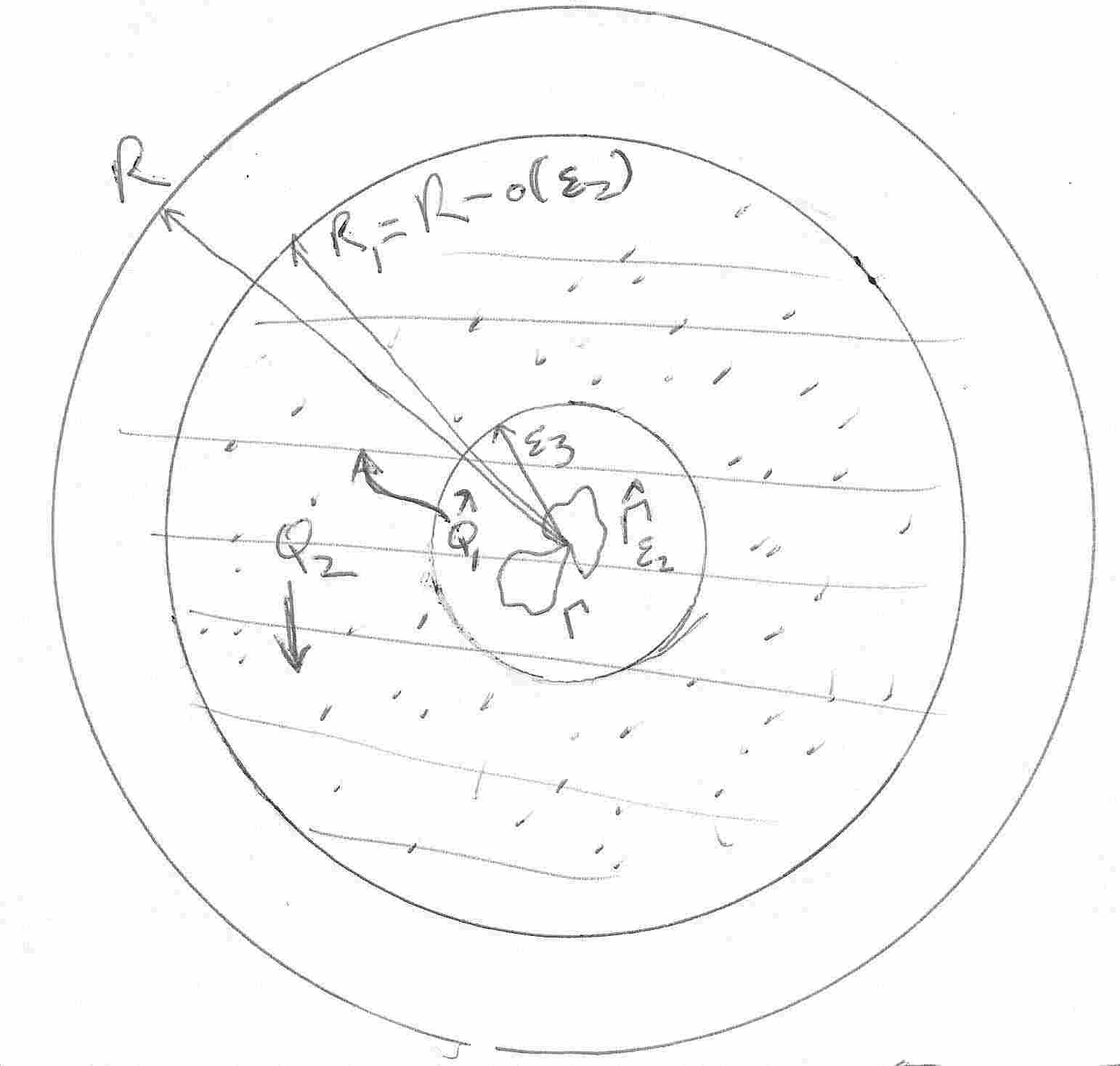}
\mycaption{Schematic of jointed-pipe regions $Q_1$ and $Q_2$.}
\label{fig:tnwire:Q}
\end{figure}

By Lemma \ref{lem:pipe} we may approximate the wire $\Gamma$ with a finitely piece-wise linear curve $\Ghat_{\e_2}$ for $\e_2>0$. For $0<\e_3<R_1:=R-o(\e_2)$ we obtain concentric jointed pipe surfaces $P_{\e_3} := P_{\e_3}(\Ghat_{\e_2})$ and $P_{R_1} := P_{R_1}(\Ghat_{\e_2})$ enclosing $\Gamma$ and lying in $N_R(\Gamma)$. Let $Q_1$ be the closed solid region bounded by the jointed pipe surface $P_{R_1}$ and the discs $D_R(\Gamma,0)$ and $D_R(\Gamma,\ell(\Gamma))$. Let $Q_2$ be the closed solid region bounded by both jointed pipe surfaces and the discs $D_R(\Gamma,0)$ and $D_R(\Gamma,\ell(\Gamma))$.  See \figref{fig:tnwire:Q}.

\begin{figure}
\picholder{4in}{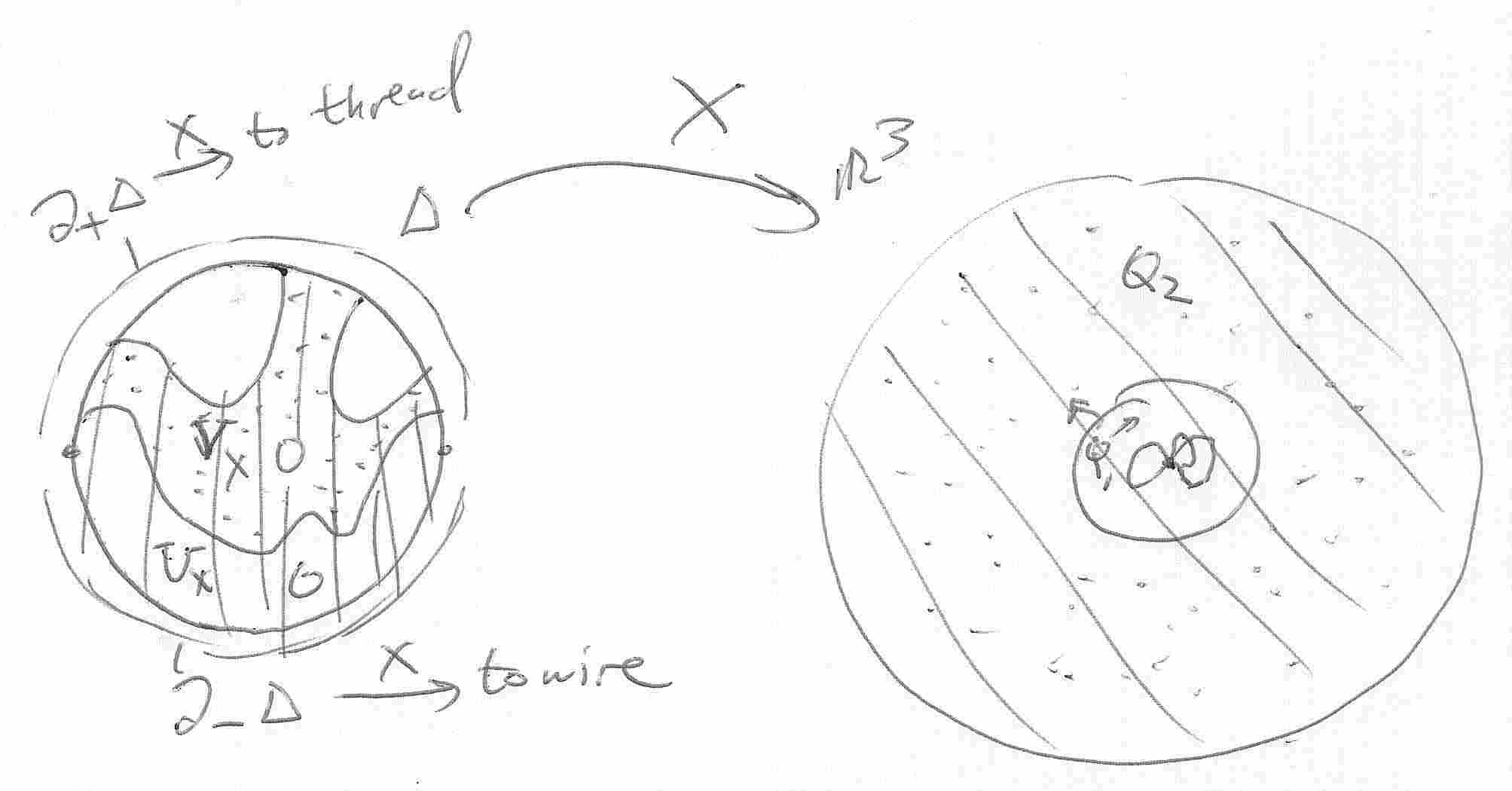}
\mycaption{Regions $U_X$ and $V_X$.}
\label{fig:tnwire:UV}
\end{figure}

At this point we will make a construction for each crescent in $M'$.  Let $X \in M'$ be a crescent.  Let $U_X$ be the connected component of $X^\n(Q_1)$ that contains $\partial_-\Delta$. Let $V_X$ be $U_X\cap X^\n(Q_2)$. Because $V_X$ lies a positive distance away from $\partial_-\Delta$, the map $X$ is real analytic in a neighborhood of $V_X$. Because the boundary $Q_2$ is piece-wise real analytic, the boundary $V_X$ consists of finitely many real analytic curves.  See \figref{fig:tnwire:UV}.

\begin{figure}
\picholder{2.5in}{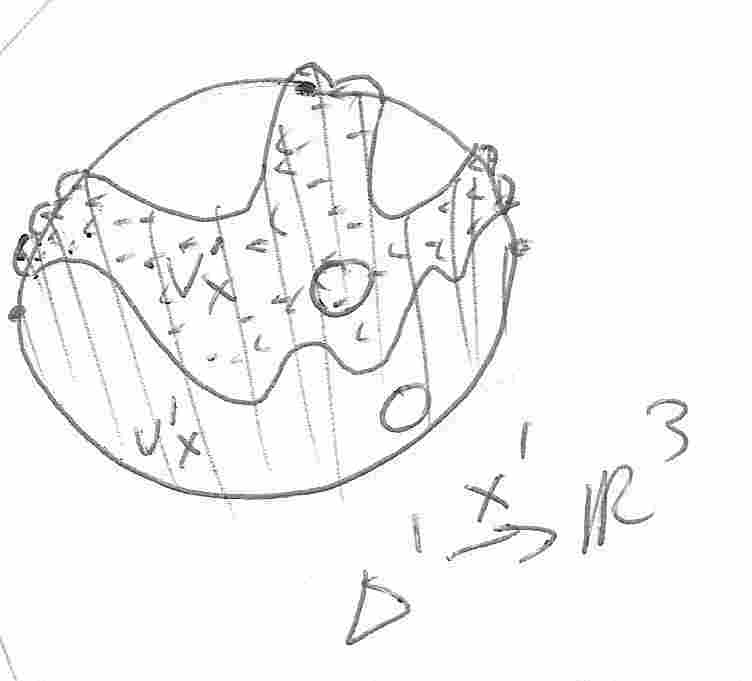}
\mycaption{Regions $U'_X$ and $V'_X$.}
\label{fig:tnwire:UVprime}
\end{figure}

The free thread curve $X|_{\partial_+\Delta}$ of the crescent $X$ is a curve with curvature given by a global constant $\kappa>0$ associated to $M$. It will be technically convenient to approximate each free thread curve using a finitely piece-wise linear curve. Pick a finite sequence of points in $\partial_+^0 \Delta$ so the first and last points mapped by $X$ to within the jointed pipe surface $P_{\e_3}$ and so that consecutive points are no more than $\e_4$ free thread arclength apart. We may arrange that this sequence includes the end points of the finitely many arcs $V_X\cap \partial_+\Delta$. Let $\sigma$ be the curve connecting the $X$ images of these points so that the curve is piece-wise linear \emph{when pulled back by the exponential map of $\Gamma$}. By Taylor's theorem the length of $\sigma$ does not exceed the length of the corresponding part of the thread by more than a factor of
\begin{equation}\label{eq:factor}
1+\kappa^2\e_4^2/2+o(\e_4).
\end{equation}
We attach semicircular regions to the outside of $\Delta$ to form a set $\Delta'$ whose new boundary parametrizes $\sigma$. We may extend $X$ to obtain $X'$ defined on $\Delta'$ so that $X'$ has area exceeding that of $X$ by no more than $\e_5/|M'|$.  Here $|M'|$ is the number of crescents in $M'$.  Let $U_X'$ and $V_X'$ be respecticely $U_X$ and $V_X$ union the semi-circular regions abutting them.  See \figref{fig:tnwire:UVprime}.

Let $R_2 = R_1 - O(\e_2^2)$ be the radius provided by Lemma \ref{lem:pipe} so that $N_{R_2}$ is enclosed by $P_{R_1}$.  We use these the above efforts to construct a subset of the rectangle $T_{R}=(0,\ell(\Gamma))\times[0,R)$ ($R<R_2$ to be determined later) which we study isoperimetrically in order to prove our theorem.  First we define $E$ to be the union of $(\im \psi\circ X'|_{V'_X}) \cap T_R$ for all $X\in M'$ and the rectangle $0\le y\le \e_6$ in $T_R$. (See \figref{fig:tnwire:E}.)  Here we choose $\e_6$ according to Lemma \ref{lem:pipe} so that the jointed pipe surface $P_{\e_3}$ lies within distance $\e_6$ of $\Gamma$.  The boundary of $E$ above $y=\e_6$ is not necessarily just the image of the thread curve.

\begin{figure}
\picholder{3.5in}{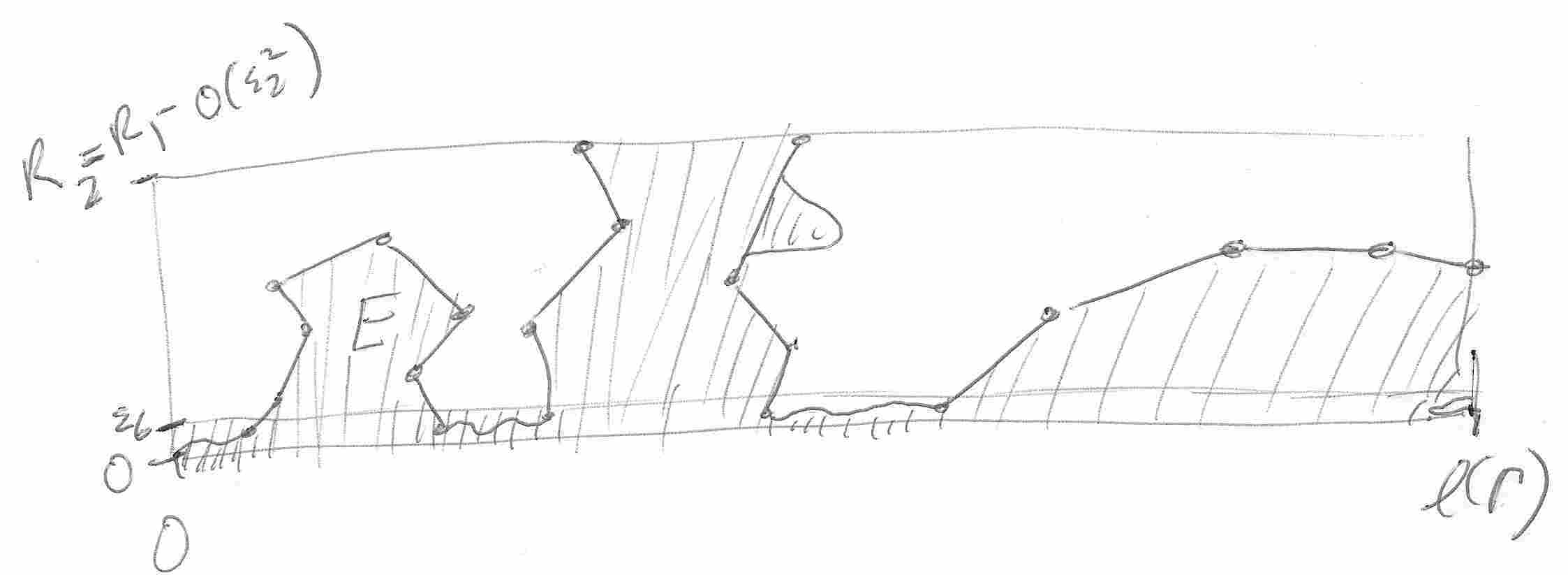}
\mycaption{Region $E$.}
\label{fig:tnwire:E}
\end{figure}

Consider the $\Psi$-image of the modified free thread curve near the wire:
$$\tau_X = \Psi(X'(\partial_+\Delta') \cap N_R(\Gamma)).$$
By \eqref{eq:nearwire:len_exp} and \eqref{eq:factor}, the $\beta$-weighted $\sH^1$ measure of $\tau_X$ is at most $\ell(X|_{\partial_+\Delta})(1+\kappa\e_4/2+o(\e_4)+ O(R^2))$ where
$$\beta(x,y)=1-4y\kmax.$$
Let $\tau_X^2$ be $\tau_X$ restricted to $y \ge \e_6$, union $\{y=\e_6\}$.  Equivalently, it is what you get when you take the part of $\tau_X$ below $y=\e_6$ and project it onto $y=\e_6$.  This operation does not increase weighted length, and we see from \eqref{eq:factor}
\begin{equation}\label{eq:threadnearwire:tau}
	\sH^1_\beta(\tau_X^2) \le \ell(X|_{\partial_+\Delta})(1 + \kappa^2 \e_4^2/2 + o(\e_4) + O(R^2)).
\end{equation}
Moreover, by construction of $X'$, we see that $\tau_X^2$ is a union of finitely many line segments.  Define $\tau = \cup_{X \in M'} \tau_X^2.$  Then
\begin{equation}\label{eq:threadnearwire:PE}
	\sH^1_\beta(\tau) \le \ell(M')(1 + O(\e_4) + O(R^2)) = (\ell(\Gamma) - \lambda)(1+ O(\e_4) + O(R^2)).
\end{equation}
\begin{figure}
\picholder{3.5in}{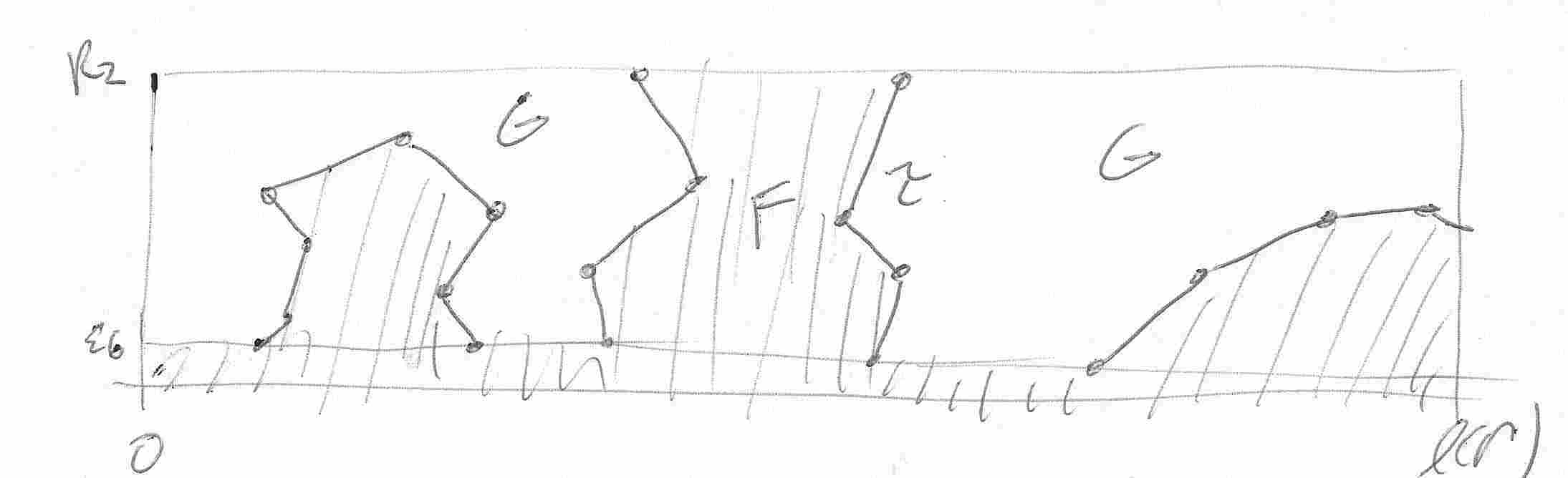}
\mycaption{Regions $F,G$.}
\label{fig:tnwire:FG}
\end{figure}

Let $F$ consist of $\tau$, $y \le \e_6$ in $T_R$, and all points $x$ in $T_R$ which are trapped by $\tau$ in the following sense: any homotopy of $X'|_{\partial_+ V_X'}$ to lie below $y=\e_6$ must pass through $x$.  (This includes in $F$ a subset of the finitely many components of $\{ y > \e_5 \} \cap T_R \setminus \tau$.)  See \figref{fig:tnwire:FG}.

\begin{lemma}\label{lem:bigclaim} We claim that $F \subset E$, and that the Lebesgue area of $F$ is bounded
\begin{equation}\label{eq:area_bound}
	A(F) \le (\kmax^\n \lambda + o(\lambda))(1 + 4\kmax R + O(R^2)) + \e_5 + \ell(\Gamma) \e_6.
\end{equation}
\end{lemma}
\begin{proof}
To see the first part of the claim: let $p$ be any point in the interior of $F$ with $y(p) > \e_6$.  (If there is no such $p$ we are done.)  Then there is some $X\in M'$ so $\Psi \circ X'$ maps $\partial \Delta'$ to a continuous curve lying in $\R^2\setminus p$ which is not null-homotopic.  The boundary $\partial \Delta'$ is contractible which means $\Psi \circ X'|_{\partial \Delta'}$ is contractible in the image of $\Psi \circ X'$; we conclude that $p \in \Psi \circ X'$.  We claim moreover that $p \in U'_X$.  If not, it lies in a connected component $C$ of $\Delta'\setminus U'_X$.  Since $U'_X$ is connected, the component is simply connected.  Given the level of regularity of our boundary (finitely piece-wise real analytic) it is not hard to see that we can retract $\Delta'\setminus C$ onto $\partial_- \Delta$ without passing through $C$.  This induces a homotopy of the $\Psi$ image of $\partial_+ \Delta' \cap U'_X$ into $[0,\ell(\Gamma)]\times\{0\}$ which by sentence 3 of this paragraph, must pass through $p$.  We get a contradiction.  The set $F$ is clearly Lebesgue measurable.  And $U_X'$ is Lebesgue measurable due to its finitely piecewise real-analytic boundary.  
We write
\begin{eqnarray}\label{eq:threadnearwire:AE}
	A(F)-\ell(\Gamma)\e_6 &\le&  \sum_{X\in M'} A(X'(U_X')) \\
	\nonumber&\le& \sum_{X \in M'} \int_{\Psi(X'(U_X'))} \sH^0(X^\n(\{y\})) d\sH^2 y\\
	\nonumber&=&\sum_{X \in M'} \int_{U_X'} (J \Psi \circ X')(x) d \sL^2(x)\\
	\nonumber&=&\e_5 + \sum_{X \in M'} \int_{U_X} (J \Psi \circ X)(x) d \sL^2(x)\\
	\nonumber&\stackrel{\eqref{eq:nearwire:len_exp}}{\le}&
	\e_5 + (1 + 4\kmax R + O(R^2)) \sum_{X \in M'} \int_{U_X'} (J X)(x) d \sL^2(x)\\
	\nonumber&=&D(M') (1 + 4\kmax R + O(R^2)) + \e_5\\
	\nonumber&\stackrel{\eqref{eq:nearwire:DM}}{\le}& (\kmax^\n \lambda + o(\lambda))(1 + 4\kmax R + O(R^2)) + \e_5.
\end{eqnarray}
Here the third implication follows from the area formula \cite[Thm 3.2.3]{federer} and $J$ denotes Jacobian.  The fifth implication follows from $X$ being conformal.
\end{proof}

We have constructed a set $F \subset T_{R}$ with boundary $\tau$; the area and perimeter are given in \eqref{eq:threadnearwire:PE}, \eqref{eq:threadnearwire:AE}.

\begin{claim} \label{clm:nomeet} For $\lambda$ small relative to the geometry of $\Gamma$, for any $\e_* \in (0,1)$, if we choose 
	\begin{equation}\label{eq:chooseR}
		R = \sqrt{\frac{2\lambda}{\kmax \pi}}(1+\e_*)
	\end{equation}
	and choose $\e_1,\ldots \e_6$ above small enough relative to $\Gamma$ and $\e_*$ then $F$ does not meet $y=R_2$.
\end{claim}

This is our key claim; it will allow us to show that $M'$ is a near-wire minimizer.  We prove the claim by doing isoperimetric analysis of the complement $G$ of $F$ in $T_{R_2}$.  Indeed, if Claim \ref{clm:nomeet} fails, then we have two cases: Case I if $G$ contains a connected component abutting both $x=0$ and $x=\ell(\Gamma)$ and Case II otherwise.  We may then apply Lemma \ref{lem:wtd_iso}.

\textsc{Case II.}  Let $G_1$ be the union of any connected components of $G$ abutting $x=0$.  Let $G_2 = G\setminus G_1$.  Let $P_i = \sH^1_\beta(\partial G_i)$.  Then $G_1$ does not abut $x=\ell(\Gamma)$ and $G_2$ does not abut $x=0$.  We know that $P_i> 0$ and $G_i \ne \emptyset$ for at least one $i$, because $R^2 = O(\lambda)$ and \eqref{eq:threadnearwire:AE} prevents $F$ from being all of $T_{R}$.  We have by Lemma \ref{lem:wtd_iso} with $m= 4\kmax + O(R^2)$ and $Y=R$,
\begin{eqnarray*}
A(G_i) &\le& \begin{cases}
	P_i R + O(R^2), & P_i < \pi R(1-mR),\\
	P_i R - \frac{\pi R^2}{2} + O(R^2), & \text{else}
	\end{cases}
\end{eqnarray*}
We must have at least one of $G_i$ be non-empty and fall into the second case; otherwise the area of $F$ will be too large relative to the reference competitor.  This gives
$$A(G) \le PR - \frac{\pi R^2}{2} + O(R^2) = \ell(\Gamma)R - \kmax^\n\lambda(1+\e_*) + O(\lambda)$$
where $P = \sH^1_\beta(\partial G) = \ell(\Gamma)-\lambda$.  On the other hand, by lemma \ref{lem:bigclaim},
$$A(G) \ge \ell(\Gamma)R - \kmax^\n \lambda + o(\lambda) + \e_5 + \ell(\Gamma)\e_6$$
Given any $\e_*$, we may pick the $\e$'s small enough to expose a contradiction.  This proves Case II.  

\textsc{Case I} follows similarly; in this case we observe that $\tau_2$ must reach $y=R_2$ by travelling up at least one of $x=0$ or $x=\ell(\Gamma)$.  Without loss of generality, assume the first.  Then we may apply the argument of Case I with $G_2 = G, G_1 = \emptyset.$

Now the finite crescent Alt competitor $M' \subset M$ was chosen arbitrarily subject only to the length rule \eqref{eq:Mprimelen}.  As such we may guarantee that $M'$ includes any given crescent.  In this way, we have shown that no crescent $X$ of $M$ has $X|_{U_X}$ touching the jointed pipe surface $P_{R_1}$.  

We now show that in fact each $X\in M$ has free thread lying in $N_{R_0}$.  At this point, the only possible problem is that the thread of $X$ might ``escape out the ends'' $D_1 := D_{R_0}(\Gamma,0)$, $D_2 := D_{R_0}(\Gamma,\ell(\Gamma))$ of our jointed pipe enclosure $Q_1$.  We show that this is impossible, as follows.  Consider the open set $Q^* = \Delta\setminus U_X$.  The free thread outside the $Q_1$ enclosure is parametrized on $\partial_+ \Delta \cap Q^*$.  Let $C$ be an arc in this closed set.  By Claim \ref{clm:nomeet}, there are two possibilities:
\begin{enumerate}
	\item $X(\partial C)$ both lie in the same disc $D_i, i=1,2$.  Then let $f:\R^3 \to R$ be the affine function containing $D_i$ in its level set and with gradient pointing out of $Q_1$.  Take the connected component $C^*$ of $X^\n(\{ f \ge 0\})$ containing $C$.  Modify $X$ on $C^*$ by orthogonally projecting to $f=0$.  This strictly reduces the Dirichlet energy of $X$ and respects the fixed boundary condition.  This contradicts the minimality of $M$.
	\item $X(\partial C)$ lie in different discs $D_1,D_2$.  Pick $R'$ small enough relative to the geometry of $\Gamma$ so this section's arguments work for both $R=R'$ and $R=2R'$.  We will find ourselves in Case II the first time.  We can then see in the $R=2R'$ run that $F$ has area at least $\ell(\Gamma) R'$ and so $M$ loses to our model competitor $P_0$.
\end{enumerate}

\subsection{Entire minimizer lies near wire}

In the previous section we showed that the thread lies in a tubular neighborhood of the wire.  In this section we show that the entire thread-wire surface lies in a slightly larger tubular neighborhood.  A general lemma about minimal surfaces suffices.

\begin{lemma} \label{lem:nearwire}
Let $X \in H^1(\Delta^o,\R^3) \cap C^0(\Delta)$ parametrize a minimal surface conformally and harmonically:
$$\Delta X = 0 \qquad \innp{X_{x_1}}{X_{x_2}} = 0 \qquad |X_{x_1}| = |X_{x_2}|.$$
Let $\Gamma:[0,\ell(\Gamma)]\to \R^3$ be a $C^2$ embedded curve, parametrized by arclength.  Assume that $X(\partial \Delta)$ lies in an $r_1$ tubular neighborhood of $\Gamma$.  Then if $D(X) \le \alpha$, the entire surface $\im X$ lies in an $r_2$ tubular neighborhood of $\Gamma$,
for
$$r_2 = C_1 r_1 + C_2 \alpha^{2/3} + o(\alpha^{2/3}).$$
\end{lemma}
\begin{proof}
Let $\kmax$ be the max of $|\Gamma''|$.
\begin{claim}\label{clm:CH} If $\gamma$ is a closed curve of length $2\ell$ in $N_r(\Gamma)$ then
$$\CH(\gamma) \subset N_s(\Gamma)$$
where
$$s = r + C(\Gamma) \ell^2 + o(\ell^2)?$$
\end{claim}
\begin{proof}
	Use Taylor's theorem to see that having a point in $\CH(\gamma)$ distance $a$ outside $N_r(\Gamma)$ requires points of $\gamma$ appearing in $N_r(\Gamma)$ at distance at least $C(\Gamma) \ell^2$ apart, measured within $N_r(\Gamma)$.
\end{proof}

\newcommand{\Nhat}{\hat{N}}
Let $\Ghat_\e$ be the piece-wise linear curve provided by Lemma \ref{lem:pipe}.  For $r < R_0(\Gamma)-o(\e)$, the pipe surface $P_r(\Ghat_\e)$ is a piece-wise real-analytic manifold, consisting of pieces of cylinders and two discs.  Adjacent pieces of cylinders correspond to adjacent segments of $\Ghat$.

	Let $f(p) = d(p,\Ghat)$.  Consider $g = f \circ X.$  It attains some maximum $\fmax$.  Our goal is to bound that maximum.  Now for every $0 < c < M := \min(R_0/2,\fmax),$ consider the level set $f^\n(\{c\})$.  This set is a finitely piecewise real-analytic curve.  If its total length of is $\ell$, then we can find a loop in the level set of length $\le \ell$ and we get by Claim \ref{clm:CH} that $\fmax \le c + C(\Gamma) \ell^2 $.  In other words, $\ell$ being small forces $\fmax$ to be small.  On the other hand, if these lengths stay large, then they force large area.  Informally,
	\begin{eqnarray}\label{eq:bigalpha}
		\alpha \ge \int_{r_1}^{M} \sH^1(f^\n(c)) \, dc & \ge& \int_{r_1}^{M} C(\Gamma) \sqrt{\fmax - c} \, dc\\
		\nonumber&\ge& C(\Gamma) \int_{r_1+\fmax-M}^{\fmax} \sqrt{\fmax-c} \, dc \\
		\nonumber&=& C(\Gamma) \int_{0}^{M-r_1} \sqrt{u} \, du
		= C(\Gamma) (M-r_1)^{3/2}
	\end{eqnarray}
	whence $M = \fmax$ because $\alpha$ is small compared to the geometry of $\Gamma$ and we get $\fmax = r_1 + C(\Gamma)\alpha^{2/3}.$
	To complete our proof, we justify \eqref{eq:bigalpha}.
	Let $U = X^\n(\overline{N_{M} \setminus N_{r_1}})$.  At almost every $p\in U$, the the surface at $X(p)$ meets the level set of $f$ transversely at a smooth portion of the level set, locally along a curve $\sigma$ with unit cotangent $\sigma^*$ at $X(p)$.  We have at $p \in U$,
$$(\cos \theta) X^*(\dvol \R^3) = \theta X^*(df) \wedge X^*(\sigma^*)$$
as $2$-covectors at $p$ where $\theta$ is the angle between the tangent plane to $X$ and the tangent plane to the level set of $f$.  This justifies the first inequality in \eqref{eq:bigalpha}.  The second inequality follows from Claim \ref{clm:CH} as described above.
\end{proof}

\subsection{Proof of weighted isoperimetric inequality}\label{sec:isoproof}
We will prove Theorem \ref{thm:nearwire} by relating a near-wire minimizer to a region in a long planar rectangle which must obey isoperimetric inequalities.  In our method, a weighted isoperimetric inequalithy arises naturally.  Without weighted isoperimetry, we are only able to show the theorem for wires with small total curvature.

\begin{figure}
\picholder{3in}{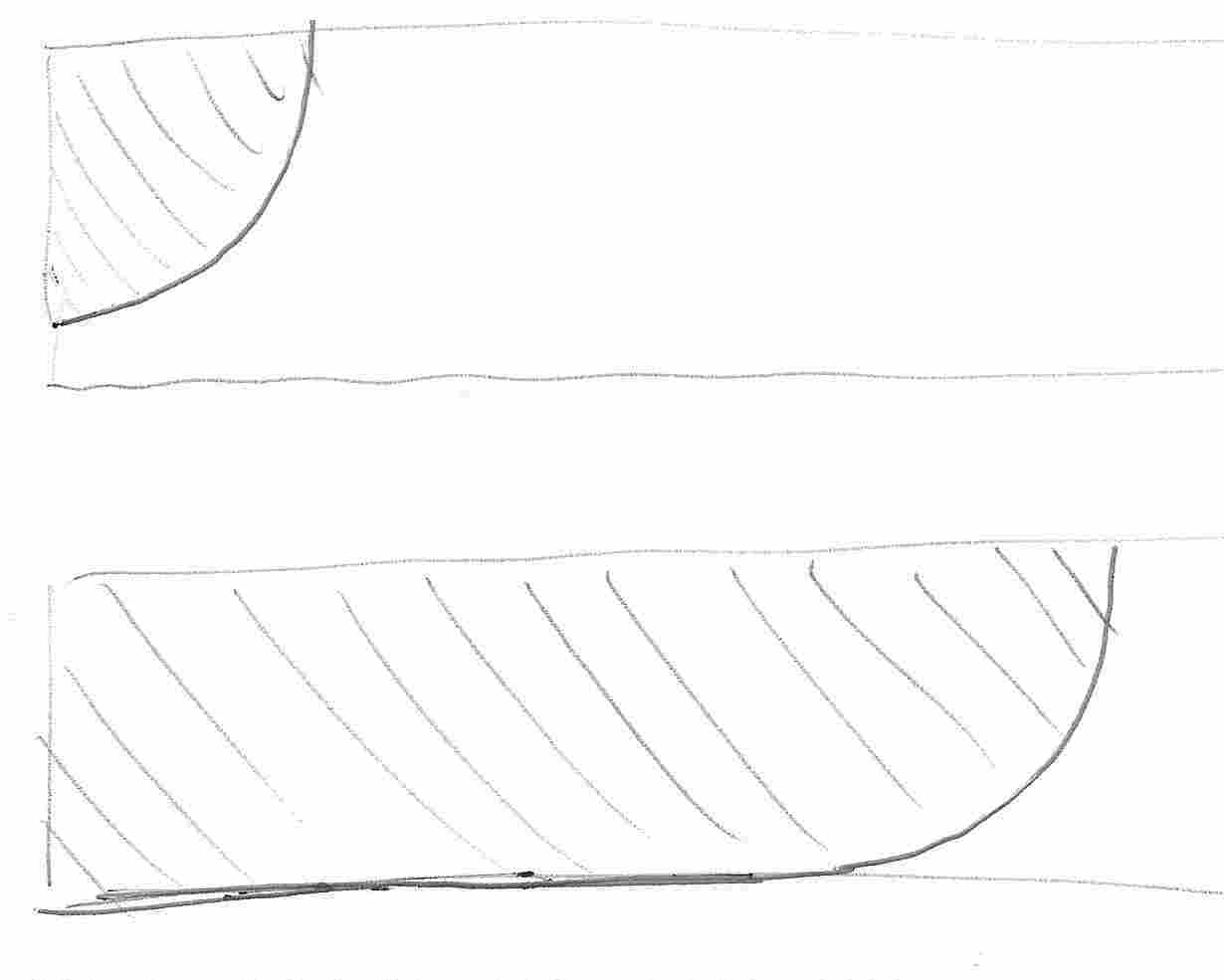}
\mycaption{Sharp examples in $m=0$ case.}
\label{fig:iso:}
\end{figure}

\begin{lemma}\label{lem:wtd_iso} Let $K \subset T_Y := \{(x,y)\st x\ge 0, 0 \le y < Y\}$ be a set with finitely piece-wise linear boundary.  Let $P$ be the $(1-my)$-weighted perimeter of $K$ in $T_Y$, where $0 \le mY < 1$.
\begin{equation}\label{eq:wtd_iso1}        
	A(K) \le \frac{P^2}{\pi}(1-mY)^{-2}.
\end{equation}
Moreover, if $P > \pi Y(1-mY)$,
\begin{equation}\label{eq:wtd_iso2}
	A(K) \le \left(P - \frac{\pi Y}{1-mY}\right)Y + \frac{\pi Y^2}{2}
\end{equation}
When $m = 0$ these are sharp; equality is achieved by the half disc in \eqref{eq:wtd_iso1} and by the half rounded-rectangle in \eqref{eq:wtd_iso2}.  See \figref{fig:iso:}.
\end{lemma}
	
\begin{proof}
	If $E$ is a Caccioppoli subset (in our case a finite polygon) in the plane then
\begin{eqnarray}\label{eq:nearwire:halfiso}
	A(E) \le \frac{1}{4\pi}\Per(E,\R^2)^2.
\end{eqnarray}
which is an equality for discs.  Equation \eqref{eq:wtd_iso1} follows from \eqref{eq:nearwire:halfiso} after extending $K$ by reflection across $x=0$ and $y=Y$.

Let $x^* \ge 0$ be the unique value such that
\begin{eqnarray*}
	\Per(K,T_Y \cap \{ x > x^*\}) &\le& \pi Y(1-mY)\qquad \text{and}\\
	\Per(K,T_Y \cap \{ x < x^*\}) & \le &P - \pi Y(1-mY).
\end{eqnarray*}

Then we may repeat the argument above to conclude that
$$A(K \cap \{ x > x^*\}) \le \pi Y^2/2.$$

Our question now reduces to showing that on $H := T_Y \cap \{x < x^*\}$
$$A(K \cap H) \le \Per(K, H) Y.$$

As an aid, we use a one-dimensional isoperimetric inequality:

\begin{claim} Given $J \subset [0,Y)$ which is union of finitely many relatively closed intervals,
	$$\int_J 1\, d \sH^1 \le Y \int_{\partial J} 1-my \, d \sH^0.$$
\end{claim}
\begin{proof}
Showing the claim means showing the non-positivity of
$$\varphi(J) = \int_J 1\, d \sH^1 - Y\int_{\partial J} 1-my \, d \sH^0 .$$
Given any connected component $[y_1,y_2]$ making up $J$, you can decrease $y_1$ at unit speed and increase $\varphi$ at rate $1 - mY > 0$.  Do this for all segments of $J$ until endpoints pile up.  Delete endpoints that have collided.  This reduces our question to considering intervals of the form $[0,b]$.  We have
$$
\varphi([0,b]) = b - Y(1 + 1-mb) = (1+mY)b - 2Y < 2(b - Y) \le 0.
$$
So we took an arbitrary $J$, possibly modified it in a way that only increases $\varphi$ and got a non-positive value.
\end{proof}

We may now show our result:
\newcommand{\xbar}{\underline{x}}
\begin{eqnarray*}
	\int_{K \cap H} 1 &\le& \int_{\xbar=0}^{x^*} \, d\xbar \, \int_{K \cap \{x = \xbar\}} 1 \, dy \\
	&\le& Y \int_{\xbar=0}^{x^*} \int_{\partial(K \cap \{x = \xbar\})} 1-my \, dy \\
	&\le& Y H^1_{\beta}(\partial (K \cap H),H)
\end{eqnarray*}

\end{proof}
\section{How planes may intersect crescents}\label{sec:plane_cresc}
\newcommand{\harmsl}{\lemLname{lem:harm_level_set:}}

In this section we investigate what the intersection between a plane and an Alt minimizer can look like.  Essentially, we show that if a connected component of the intersection contains finitely many wire points, then that component has the structure of a finite graph.  Moreover, this graph can only touch the thread curve in one point.  We state the full lemma below.  The full statement is more technical.  It only requires knowledge about how a compact piece of the plane intersects the wire.  Moreover, there is a technical issue which arises in the case that the free thread curvature is zero.

\begin{lemma} Let $\Gamma$ be an embedded nonplanar $C^1$ wire curve.  Let $V$ be a plane in $\R^3$.
\begin{enumerate}
\item[$(a)$] Let $W$ be a compact subset of $V$ which $\Gamma$ intersects at most a finite number $m$ times.
\item[$(b)$] Let $(X,\phi_-)$ be an Alt crescent with $\im X$ disjoint from $\partial_V W$ (the boundary of $W$ in the topology of $V$).
\end{enumerate}
(See \figref{fig:plane_cresc:intg}.)  Then the pre-image $X^\n(W)$ has at most $m$ connected components.  Each connected component is either
\begin{enumerate}
	\item a single point of $\partial_-\Delta$, or
	\item a finite tree graph
		\begin{enumerate}
			\item with all interior nodes having even valence of at least $4$,
			\item with at least one node on $\partial_- \Delta$,
			\item with at most one node on $\partial_+ \Delta$.
		\end{enumerate}
	\item a set containing all of $\partial_+ \Delta$.  Moreover, in this case, the free thread curvature of $(X,\phi_-)$ is zero.
\end{enumerate}
	In particular, if $\kappa >0$ and $m=0$, then the Alt minimizer does not touch the set $W$.  If $\kappa > 0$ and $m\le2$, then the pre-image $X^n(W)$ has no interior nodes; this implies that the interior of the Alt crescent ($X|_{\Delta^o}$) never oscullates the set $W$.
\end{lemma}

In this lemma, case $(ii)$ is by far the most important.  Case $(iii)$ is only relevant in the special case where the free thread curvature is zero and the free thread consists of straight segments.

\begin{figure}
\picholder{5in}{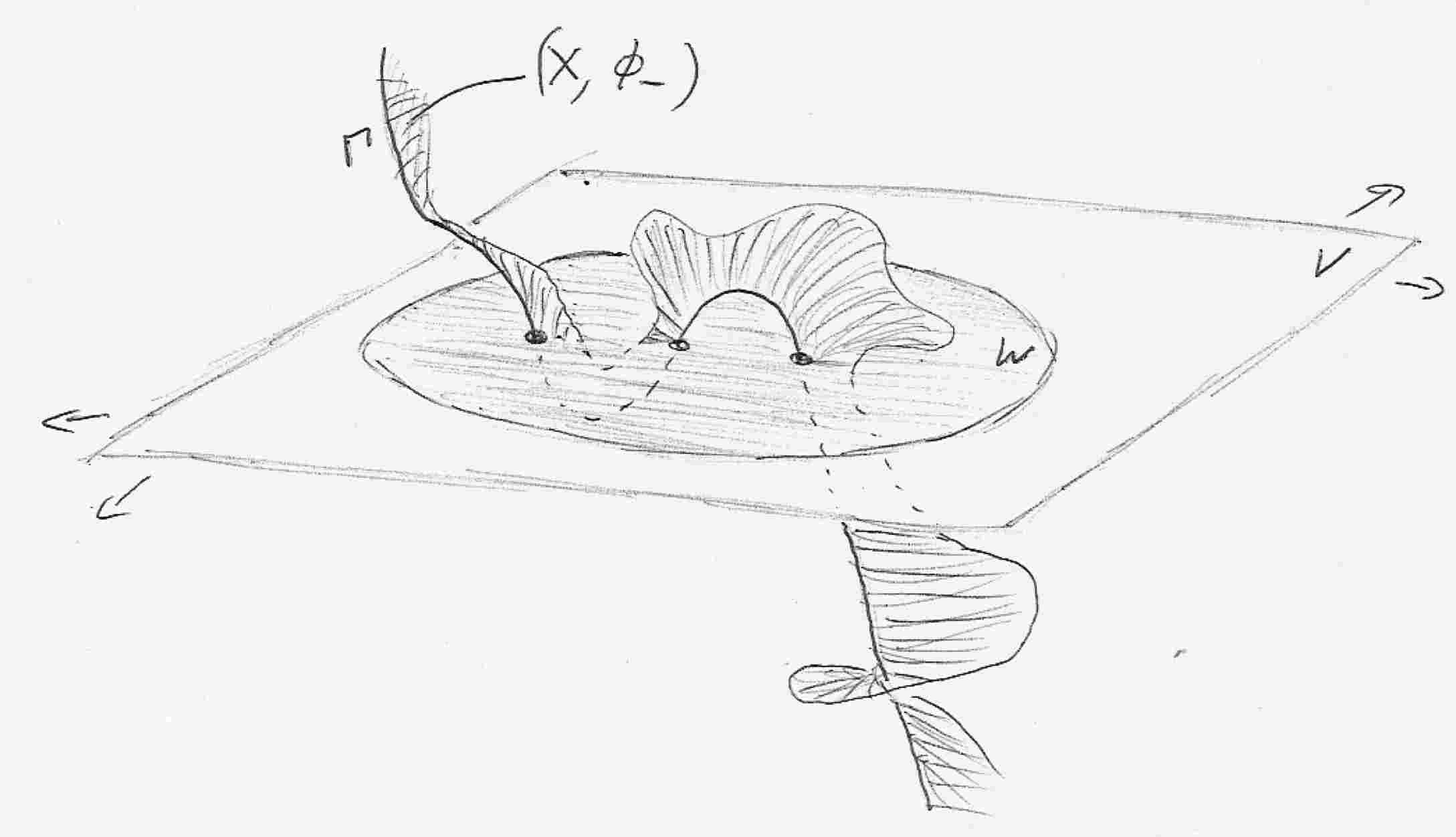}
\caption{Conditions for Lemma \ref{lem:plane_cresc:intg}.}
\label{fig:plane_cresc:intg}
\end{figure}

We may reduce the proof of this lemma to a statement about the level sets of harmonic functions on the unit disc.  Below we prove the relevant lemma.  Then we prove the \lemMLname{lem:plane_cresc:intg}.

\startevenpage
\section{Level sets of harmonic functions on the disc}\label{sec:harm_level_set}
\begin{wrapfigure}{o}{1.8in}
\DeltaNotation{plane_cresc}
\end{wrapfigure}

Consider a minimal surface $Y:\Delta \to \R^3$ spanning a contour.  Let us cut this surface with a plane, expressed as $F=a$ for $F$ a linear function on $\R^3$.  Assume that $Y$ is parametrized conformally and harmonically on $\Delta^o$, so as to minimize Dirichlet energy.  Then we may pull back a linear function $F$ by $Y$ to obtain a function $F \circ Y$ on $\Delta$ which is harmonic on $\Delta^o$.  The intersection of the minimal surface with the plane is parametrized by $Y$ on the subdomain $F \circ Y = a$.

In this context, Rad\'o's lemma helps us understand intersections between planes and minimal surfaces.

\begin{lemma} 
\textbf{[Rad\'o's lemma]}\footnote{See \cite[p. 272]{dhkw1}.  For more about Rad\'o's work, see \cite{rado}.} If $h:\Delta \to \R$ is harmonic in $\Delta^o$ and if its derivatives vanish to orders $0,1,\ldots m$ at some point $p \in \Delta^o$, then $h$ changes sign on $\partial \Delta$ at least $2(m+1)$ times.
\end{lemma}

Roughly speaking, the idea of Rad\'o's lemma is that the level set $h=0$ looks like a graph.  The graph cannot have cycles (closed loops) because then by the {Maximum Principle} an entire open set would have $h=0$, whence by analytic continuation $h$ vanishes on all of $\Delta$.  Since the graph does not have cycles, we expect that an interior zero of order $m+1$---which gives a node with valence $2(m+1)$---propogates outward to force at least $2(m+1)$ sign changes on the boundary.  The type of result we need is similar, but it has a special condition on the top boundary $\partial_+ \Delta$ of the unit disc.  There we assume that $h$ does not achieve any strict local extrema.  Under that assumption, we are able to guarantee a certain number of sign changes on the lower boundary $\partial_- \Delta$.

The following lemma characterizes a level set $h=a$ of a harmonic function $h$ on the unit disc $\Delta$.  It is written in a form that allows it to be applied when we know properties of $h$ only for a part $U$ of the level set.  To understand the essence of the lemma, the reader may find it helpful to read it in the case that $U$ is the entire level set.  Our final preparation is to clarify some notation in \figref{fig:plane_cresc:delta_notation} and the following definition.

\begin{definition}
	A \emph{(planar) graph} is a set of points $V$ (\emph{nodes})in $\R^2$ and a set of continuous curves (\emph{edges}) from $V$ to $V$.  We allow multiple edges to connect the same pair of nodes and to connect a node to itself.  The \emph{valence} of a node is the number of edges emanating from it.  We assume that every node has valence at least $1$.  A graph is a \emph{tree graph} if it is connected and simply connected.
\end{definition}

\begin{lemma} \textbf{[Harmonic Level Set]}\label{lem:harm_level_set:} Let $h\in C^0(\Delta,\R)$ be harmonic and real-analytic on $\Delta \setminus \partial_-\Delta$.  Let $U$ be a nonempty union of connected components of a level set $h=a$ on $\Delta$.  Assume that:
\begin{enumerate} 
\item[$(a)$] The function $h$ is nonconstant on $\partial_+ \Delta$ and does not attain any local extrema on the domain $\Delta$ at points of the set $\partial_+^o \Delta$.  In other words, for each point $p$ in $\partial_+^o \Delta$ and each neighborhood $\sN$ of $p$ in $\Delta$,
	$$\min_\sN h < h(p) < \max_\sN h.$$
\item[$(b)$] We have $h=a$ at only $m$ points of $U \cap \partial_- \Delta$.
\end{enumerate}
Then $U$ consists of at most $m$ connected components.  Each component $P$ is either a single point on $\partial_- \Delta$ or is a planar graph
\begin{enumerate}
	\item which is a finite tree graph,
	\item with all interior nodes having even valence of at least $4$,
	\item with at least one node on $\partial_- \Delta$,
	\item with at most one node on $\partial_+ \Delta$.
\end{enumerate}
\end{lemma}

In this section we prove the Harmonic Level Set Lemma by proving several supporting lemmas.

\begin{lemma}\label{lem:harm_level_set:nocycle} Let $h\in C^0(\Delta,\R)$ be a nonconstant function which is harmonic and real-analytic on $\Delta \setminus \partial_-\Delta$.  Then
	\begin{enumerate}
		\item[$(i)$] the level set $h=a$ cannot contain a Jordan curve.
	\end{enumerate}
Moreover, if assumption $(a)$ of \lemLname{lem:harm_level_set:} holds then:
	\begin{enumerate}
		\item[$(ii)$] the level set $h=a$ cannot contain a curve $\gamma:[0,b] \to \Delta$ mapping $(0,b)$ to $\Delta^o$ and $0,b$ to $\partial \Delta$.
	\end{enumerate}
\end{lemma}
\begin{proof}
	If the level set did contain a Jordan curve, then the interior of the Jordan curve would be an open set $V$ for which $h=a$ on $\partial V$.  Then by the {Maximum Principle}, we would have $h\equiv a$ on $V$.  By analytic continuation we get $h\equiv a$ on $\Delta$, contrary to assumption.  So the level set did not contain a Jordan curve in the first place.

	As for $(ii)$, we again have an open set $V$ formed by the curve and the top boundary $\partial_+ \Delta$.  We again apply the {Maximum Principle}.  To avoid the contradiction that threatened to occur in the previous paragraph, an extremal value other than $a$ must be attained by $h$ on $\partial V$.  This means that $\partial V$ attains an extremal value for $V$ on $\partial V \cap \partial_+^o \Delta$.  But then assumption $(a)$ of the \lemMLname{lem:harm_level_set:} has been violated.
\end{proof}

\begin{lemma}\label{lem:harm_level_set:graph} Let $h\in C^0(\Delta,\R)$ be harmonic and real-analytic on $\Delta \setminus \partial_-\Delta$.  Let $P$ be a connected component of the level set $h=a$ on $\Delta$ which intersects $\partial_- \Delta$ at only finitely many points.  Then $P$ is a planar graph in the following sense.  Let $N$ be the set of critical points of $h$ in $\Delta\setminus \partial_-\Delta$.  Let $B_-$ be the (finite) set of places where $h=a$ on $\partial_- \Delta$.  Let $B_+$ be the places where $h=a$ on $\partial_+^o \Delta$ and $\nabla h \ne 0$.  Let $E$ be the remainder of $P$.  Then $E$ is a disjoint union of continuous curves with ends in $N \cup B_ \cup B_+$.  The curves are real-analytic on their interiors, and they remain real-analytic up to any end points lying in $N \cup B_+$.  The valence of any node of $P$ in $\Delta^o$ is at least $4$.
\end{lemma}
\begin{proof}
	Let $E$ be the subset of $P\setminus \partial_- \Delta$ where $\nabla h$ is non-zero.  Let $N$ be the remainder of $P$, where $\nabla h$ vanishes.  For each $n \in N$ not lying in $\partial_- \Delta$, we examine the convergent expansion for $h$ near $n$.  The lead term must be a homogeneous harmonic polynomial\footnote{The $\text{Real()}$ function extracts the real part of a complex number.} $\text{Real } (az^k)$, for $a$ a nonzero complex number and $k \ge 2$.\footnote{We cannot have $k=0$ because then by analytic continuation, $h$ would be constant on all of $\Delta$; this would violate both assumptions $(a)$ and $(b)$ of the \lemLname{lem:harm_level_set:}.}  It is easy to show that in some ``nodal'' neighborhood $U_n$ of $n$, $P$ has the structure of a graph which consists of $2k \ge 4$ edges, each connecting $n$ to a point in $\partial U_n$.  Now consider a point $e \in E$.  For $\e$ a small positive value, let $W^\e$ be the set of points $x$ of $\Delta$ not lying in any nodal neighborhood $U_n$ and having $(\nabla h)(x) $ exceeding $\e$.  For sufficiently small $\e$, there is a connected component $C_e^\e$ of $E \cap W^\e$ containing $e$.  Applying the Implicit Function Theorem to $h$ at each point of $\overline{C_e^\e}$ and taking a finite subcover, we conclude that $C_e^\e$ is a continuous curve from $\partial W^\e$ to itself.  As we let $\e$ decrease, we obtain extensions of this curve.  For each end of the curve, one of two things happens.
	\begin{enumerate}
		\item Either at some value of $\e$ the curve touches a boundary of a nodal neighborhood $U_n$ for some $n$.  In this case we can use our analysis of $h$ in $U_n$ to demonstrate that the curve in this direction connects the point $e$ to the node $n$.
		\item Case $(i)$ never occurs as $\e$ goes to zero.  In case $(ii)$, we have that the curve terminates closer and closer to $\partial_- \Delta$.  By the compactness of $\partial \Delta$, it must have limit point(s) on $\partial \Delta$.  Can it have more than one?  Say it did, at $q_1,q_2$.  Then by assumption $(b)$ of \lemLname{lem:harm_level_set:}, we can draw a segment on $\Delta$ from $\partial_+^o \Delta$ to a point $q^*$ in $\partial_-^o \Delta$ between $q_1$ and $q_2$ with $h(q^*) \ne a$.  By construction, the curve $C_e^\e$ crosses this segment at an infinite sequence of distinct places as $\e$ goes to zero.  These intersections must accumulate somewhere on the segment.  But they cannot: not at $q^*$ because $h(q^*) \ne a$ and $h$ is continuous; not elsewhere on the segment by real analyticity of $h$.  So the curve $C_e^\e$ approaches a unique limit point as $\e$ goes to zero.
	\end{enumerate}
	We have shown that $P$ consists of interior nodes $N$, boundary nodes $B_-,B_+$, and a set $E$ which decomposes into continuous curves connecting these nodes.  Curves which connect $N \cup B_+$ to $N \cup B_+$ stay a positive distance from $\partial_- \Delta$ and so are real-analytic.
\end{proof}

\begin{lemma}\label{lem:harm_level_set:boundary_path} Let $h\in C^0(\Delta,\R)$ be harmonic and real-analytic on $\Delta \setminus \partial_-\Delta$.  Let $U$ be a union of some connected components of the level set $h=a$ which intersects $\partial_- \Delta$ in at most finitely many places.  Consider the set $Z$ defined by either $h \ge a$ in $U$ or $h \le a$ in $\Delta$.  Let $p$ be any point on $\partial Z$ in $\Delta^o$.  Then there is a continuous simple curve $\gamma:[-1,1] \to \partial Z$ which passes through $p$ at $\gamma(0)$ and intersects $\partial \Delta$ at two places, $\gamma(1)$ and $\gamma(-1)$.
\end{lemma}
\begin{proof}
	We define the curve $\gamma$ iteratively.  For $k=1,2,\ldots$, consider $Z$ restricted to the closed disc $L_k = B(0,1-1/(2k))$ centered on the origin.  Say that $L_{k_0}$ contains $p$.  Then for $k \ge k_0$, the set $L_k \cap \partial Z$ has a graph structure inherited from that of $P$.  We may trace $\partial Z$ from $p$ in either direction, following an edge to a new node in each step.  Operating in this way we can never encounter a node we've already been to, for that would imply a closed loop in the $h=a$ level set and would violate \lemLname{lem:harm_level_set:nocycle}.$(i)$.  Moreover, our operation must end after finitely many steps because there are only finitely many nodes in $L_k$.  (Indeed, nodes are zeroes of $\nabla h$, which is real analytic up to the boundary of $L_k$.)  In this way we can define a curve $\gamma_k:[-(1-1/(2k)),1-1/(2k)] \to \partial Z \cap L_k$ which is continuous and non-self-intersecting and has $\gamma_k(0) = p$.  Moreover, we can define such a curve so $\gamma_{k+1}$ extends $\gamma_k$.  Taking $\gamma$ to be the limit of such curves, we obtain a curve defined on $(-1,1)$.  Now consider the sequence $\gamma(1-1/(2k))$.  It must approach a point $b^*$ of $B$ arbitrarily closely.  Moreover, it cannot have two points of $B$ as limit points.  Indeed, construct a circle $C^*$ about $b^*$ cutting it off from the other points of $B$.  Then $\gamma$ cannot intersect $C^*$ infinitely many times, because these intersection points would have to accumulate and they can't (not in the interior of $\Delta$ by real analyticity of $h$, and not at $C^* \cap \partial \Delta$ because $h\ne a$ there).  We conclude that $\gamma(s)$ stays inside $C^*$ after sufficiently large $s$, and converges to $b^*$.  Similarly we can show that $\gamma(s)$ converges as $s$ goes to $-1$.
\end{proof}

\begin{lemma}\label{lem:harm_level_set:fullbound} Let $\gamma,p,Z$ be as in \lemLname{lem:harm_level_set:boundary_path}.  Additionally assume that \lemLname{lem:harm_level_set:nocycle}.$(ii)$ holds.  Then the curve $\gamma$ may be extended to a curve $\gamma_2$ which exhausts the component of $\partial Z$ containing $p$ and has at most one endpoint in $\partial_+ \Delta$.
\end{lemma}
\begin{proof}
	Let $C_Z$ be the component of $\partial Z$ containing $C$.  Applying \lemLname{lem:harm_level_set:boundary_path} to each point of $C_Z \cap \Delta^o$, we get many curves lying in $\partial Z$.  Observing \lemLname{lem:harm_level_set:nocycle}, we see that they must join together to form a curve which does not self-intersect, touches $\partial_- \Delta$ at most once in each of the finitely many points of $B$, and touches $\partial_+ \Delta$ at most once.
\end{proof}

We may now marshall our supporting lemmas to prove the main analytic result of this section: \lemLname{lem:harm_level_set:}.

\begin{proof} \emph{(of Lemma \ref{lem:harm_level_set:})}
	We consider the component $P$ defined in the lemma.  There are several cases.
	\begin{enumerate}
		\item The component $P$ does not venture into the interior of the unit disc.  Then it is a closed arc of $\partial \Delta$, possibly degenerate.  If its intersection with $\partial_+^o \Delta$ is an arc of positive length, then by analytic continuation we can show that $h$ is constant on $\partial_+\Delta$; this contradicts assumptsion $(a)$.  Combining this observation with assumption $(b)$ of the \lemMLname{lem:harm_level_set:}, we see that $P$ must be a point in $\partial \Delta$.  It cannot be a point in $\partial_+^o \Delta$, because then if we study the expansion for $h$ at $P$, we see that the only way to achieve a single point level set is for $h$ to have $\nabla h$ non-zero and be perpendicular to $\partial \Delta$.  But that would then violate assumption $(a)$ of the \lemMLname{lem:harm_level_set:}.  So we conclude that if the component $P$ does not contain points of $\Delta^o$; it must consist of a single point of $\partial_- \Delta$.
		\item The alternative is that $P$ contains a point in $\Delta^o$.  By \lemLname{lem:harm_level_set:graph}, the set $P$ has a graph structure.  We demonstrate that it in fact is a finite tree graph.  First pick an edge point $e \in P \cap \Delta^o$ and apply \lemLname{lem:harm_level_set:fullbound}.  For $\e > 0$, we can consider the connected component $P_\e$ of $e$ in $P$ restricted to the shrunken unit disc $B(0,1-\e)$.  Because of the real analyticity of $h$, edges and nodes of $P_\e$ cannot accumulate; therefore, it has the structure of a finite graph.  Applying \lemLname{lem:harm_level_set:nocycle} we see that $P_\e$ does not have cycles, and so is a tree graph.  Moreover, by applying \lemLname{lem:harm_level_set:boundary_path} we are able to take $P_\e$ and augment it by extending each node on $\partial B(0,1-\e)$ and each edge exiting through $B(0,1-\e)$ by a path which reaches a point on $\partial \Delta$.  These paths do not intersect each other or $P_\e$, at peril of violating the first part of \lemLname{lem:harm_level_set:nocycle}.  Also, the augmented $P_\e$ touches $\partial_+ \Delta$ at most once, at peril of violating the second part of \lemLname{lem:harm_level_set:nocycle}.  In this way we see that $P_\e$ has been augmented by adding at most $m+1$ paths where $m$ is the finite number of times that $h$ attains $a$ on $\partial_- \Delta$.  Since $P_\e$ is a tree graph with interior node valence of at least $4$ (see \lemLname{lem:harm_level_set:graph}), this means that $P_\e$ has $n_N$ nodes and $n_E$ edges bounded like
			$$1 + 3n_N \le m+1 \qquad 1 + 4n_N \ge 2n_E.$$
			For $\e' < \e$, the graph $P_{\e'}$ extends the graph $P_\e$.  But the number of nodes and edges of $P_\e$ is uniformly bounded.  So for sufficiently small $\e$, the augmented graph of $P_\e$ has no nodes of $P$ on the augmenting paths.  We thus show that $P$ is a finite tree graph.  Moreover, it must touch $\partial_+ \Delta$ at most once.  This then forces $P$ to touch $\partial_- \Delta$ at least once.
	\end{enumerate}
	This completes our proof of \lemLname{lem:harm_level_set:}.
\end{proof}

With the \lemLname{lem:harm_level_set:} in hand, we may prove the main geometric result of this section.

\begin{proof} \emph{(of Lemma \label{lem:plane_cresc:intg})}
	Consider an Alt crescent $(X,\phi_-)$.  Define the function $h=F\circ X$.  It is harmonic (and therefore real-analytic) on $\Delta^o$ because $F$ has constant derivative and $X$ is harmonic.  We get that it is harmonic and real-analytic on $\partial_+^o \Delta$ because the Alt minimizer is real-analytic on the interior of the free thread and can be extended real-analytically across the boundary as a minimal surface (\thmLname{thm:thread:reg}).  If $h$ is constant on $\partial_+ \Delta$ then the free thread lies in a plane.  This means it has torsion $T_\gamma \equiv 0$.  If the free thread curvature $\kappa$ is nonzero, then we may look at the expansion of the surface at a non-branch point on the interior of the thread and show that the surface is locally planar.  By analytic continuation the whole Alt crescent is planar, contrary to assumption.  We conclude that if $h$ is constant on $\partial_+ \Delta$ then the free thread curvature is not positive; by \lemLname{lem:firstvar:poskappa} we must have $\kappa=0$.  This establishes conclusion $(iii)$ of the lemma we are proving.  Otherwise, we may assume that $h$ is nonconstant on $\partial_- \Delta$.

	Next we show that conclusion $(ii)$ of \lemLname{lem:harm_level_set:nocycle} holds.  If there were any path $\gamma$ in th level set $h=a$ beginning and ending in $\partial_+ \Delta$, with its interior lying in $\Delta^o$, then its image $X \circ \gamma$ is a curve lying in the plane $F=a$.  Let $U$ be the region of with nonempty interior defined by $\gamma$ and $\partial_+ \Delta$.  By \thmLname{thm:chull:tws}, the piece of surface $X|_U$ is planar.  By analytic continuation, the whole crescent is planar, contrary to assumption.  So we see that conclusion $(ii)$ of \lemLname{lem:harm_level_set:nocycle} holds.  It suffices to show that this condition holds, instead of showing that condition $(a)$ of \lemLname{lem:harm_level_set:} holds.  The reason this suffices is that the proof of \lemMLname{lem:harm_level_set:} only depends on the conclusion \lemLname{lem:harm_level_set:nocycle}.$(ii)$.

	Condition $(b)$ of the \lemMLname{lem:harm_level_set:} is already met by assumption $(a)$ of \lemLname{lem:plane_cresc:intg}.  We have thus confirmed that the function $f$ satisfies the conditions of \lemMLname{lem:harm_level_set:}.  Now we define $U$ to be the preimage $X^\n(W)$.  We would like to show that $U$ is a union of connected components of the level set $f^\n(a) = X^\n(V)$.  It suffices to show for each $q\in U$ that the connected component $C_q'$ of $q$ in $U$ is the same as the connected component $C_q$ of $q$ in the entire level set.

We have that 
$$C_q = \bigcap_S S$$
where $S$ are sets in $X^\n(V)$ which contain $q$ and are simultaneously open and closed (open-and-closed) in the $X^\n(V)$ topology.  Now consider the operations on subsets $S \subset X^\n(V)$,
\begin{eqnarray*}
\pi_1(S) &=& S \cap X^\n(W)\\
\pi_2(S) &=& S \cap X^\n(W \setminus \partial_V W).
\end{eqnarray*}
The first sends closed subsets to closed subsets as $W$ is compact; the second sends open subsets to open subsets, because
$$\pi_2(S)= S \cap X^\n(Q)$$
where $Q$ is any open set in $\R^3$ intersecting $V$ as $W \setminus \partial_V W$.  Because we assume that $X^\n(\partial U)$ is empty, $\pi_1$ and $\pi_2$ are actually the same operation, which we can call $\pi$.  The map $\pi$ sends sets open-and-closed subsets of $X^\n(V)$ to possibly smaller open-and-closed subsets of $X^\n(V)$.  Moreover, if $S$ contains $q$ then $\pi(S)$ will contain $q$.  We obtain
$$C_q = \bigcap_{S'} S' = C_q'$$
where $S'$ are open-and-closed subsets of $X^\n(W)$ which contain $q$.  This confirms that $C_q=C_q'$.  And so we know that the set $U$ defined above is indeed a union of connected components of the level set $f^\n(a)$.

Having verified that the conditions of the \lemMLname{lem:harm_level_set:} are all met, we may now employ its conclusions about $U$, which is a union of connected components of $X^\n(V)$.  They are exactly sufficient for our purposes.
\end{proof}
\section{Slice-wise parametrization}
\begin{proof} \emph{(of Lemma \ref{thm:slope:slicewise})}
	When the wire curve is planar, the Alt crescent is planar by Theorem \ref{thm:chull}.  This case is easy; in the remainder of the proof we assume that $\Gamma$ is nonplanar.  Let $\hat{s}$ be the extension of the arclength parameter of $\Gamma$ constantly along normal discs.  By the near-wire assumption, each crescent $X$ of the minimizer lies in the union of normal discs, $\dom \hat{s}$.  We have
		$$\im X \subset \dom \hat{s}.$$
		Pulling back the extended arclength parameter function $\hat{s}$ by $X$ decomposes the domain $\Delta$ of $X$ into connected level sets.  Only values $\hat{s} \in [s_0,s_1]$ occur.  The level sets for $\hat{s} \in (s_0,s_1)$ are continuous curves of positive length.  The level sets $\hat{s} = s_0,s_1$ are points $(1,0)$ or $(-1,0)$.        We prove this lemma by applying \lemLname{lem:plane_cresc:intg}.  For each $s \in \dom \Gamma$, we consider the normal disc $D(s)$.  This is a compact subset of a plane, and it intersects the plane in exactly one point, $\Gamma(s)$.  The Alt crescent $(X,\phi_-)$ is disjoint from the circle bounding $D(s)$ because it lies strictly within the tubular neighborhood $\Tub_R \Gamma$.  By \lemLname{lem:plane_cresc:intg}, the set $X^\n(D(s))$ is either\nopagebreak
\begin{enumerate}
\item a single point $q \in \partial_- \Delta$, or\nopagebreak           \item a connected set whose only component is a finite tree graph.  This graph can only touch $\partial_-\Delta$ at one point: $\phi_-^\n(s)$.  The properties $(a)$-$(c)$ listed under item $(ii)$ in \lemLname{lem:plane_cresc:intg} force the graph to be a segment connecting $\phi_-^\n(s)$ in $\partial_-\Delta$ to a point in $\partial_+ \Delta$.
\item a set which contains $\partial_+ \Delta$; moreover in this case we also have that the free thread curvature vanishes.  But that means that the entire free thread for this crescent lies in the normal disc $D(s_0)$.  This normal disc only intersects the wire at $\Gamma(s_0)$.  So we have $\Gamma(\phi_-((-1,0))) = \Gamma(\phi_-((1,0)))$ which violates the embeddedness of $\Gamma$.  We conclude that case $(iii)$ cannot occur.
\end{enumerate}
Thus we see that each level set contains a point of $\partial_- \Delta$.  The map $\phi_-$ gives a bijection between $\partial_- \Delta$ and $[s_0,s_1]$; thus we see that $\Delta$ decomposes into level sets of $X^\n(\hat{s})$ for $s \in [s_0,s_1]$.  We claim that item $(i)$ cannot occur for $q \in \partial_-^o \Delta$; indeed in that case we could decompose $\Delta \setminus \{q\}$ into two non-empty open sets $(X \circ \hat{s})^\n([s_0,s))$ and $(X \circ \hat{s})^\n((s,s_1])$.  But a disc minus a boundary point is connected!  So case $(i)$ can only occur for $q = (1,0), (-1,0)$.
Moreover, it must occur for each of these points; if the level set $(X \circ \hat{s})^\n(s_0)$ were a curve from $(-1,0)$ to $\partial_+^o \Delta$ then we would violate \lemLname{lem:harm_level_set:nocycle}.
\end{proof}
\section{Appendix}
\begin{wrapfigure}[10]{o}{0pt}
\picholder{2in}{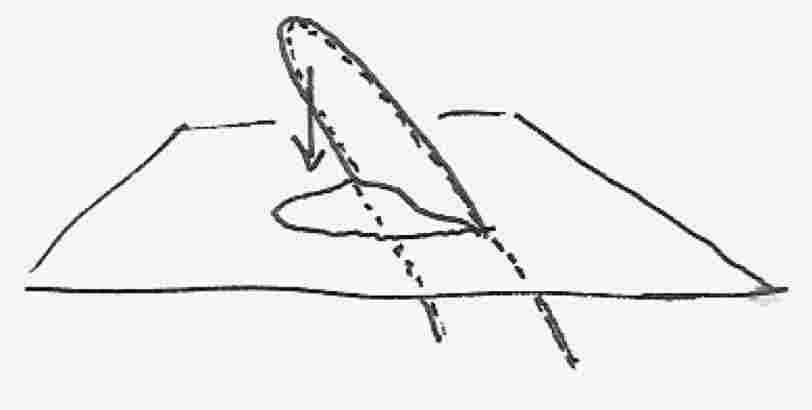}
\label{sec:firstvar}
\caption[An Alt competitor with $\kappa < 0$ may be improved]{An Alt competitor with $\kappa < 0$ may be improved.}
\label{fig:firstvar:negk}
\end{wrapfigure}

\subsection{Positivity of thread curvature}

\begin{theorem}\label{lem:firstvar:poskappa} If $M$ is an Alt minimizer for the problem $\sP(\Gamma,L)$, then it has free thread curvature $\kappa \ge 0.$
\end{theorem}
\begin{proof}
	This is straightforward.  If the thread has negative curvature, we can find a point $p\in \partial_+^o \Delta$ which is not a branch point.  We may then pick a plane perpendicular to the side normal to the surface at $X(p)$ and translate the plane towards the surface a small amount.  Then we have a situation like the one shown in \figref{fig:firstvar:negk}.  Projecting the thread and surface onto that plane reduces the Dirichlet energy of the map $X$, and it also reduces the length of the free thread.  In this way we show that there is another Alt competitor with strictly less Dirichlet energy.  This contradicts the minimizing property of $M$.
\end{proof}

\ \\

 \ \\

\subsection{Convex Hull}\label{sec:chull}
\begin{theorem}\label{thm:chull} Let $M$ be an Alt minimizer for the thread problem $\sP(\Gamma,L)$.  Then for every crescent $(X,\phi_-)\in M$ lies in the convex hull of its supporting wire $\im \phi_-$.
\end{theorem}
\begin{proof}
	It will suffice to show that the thread curve lies in the convex hull of the wire curve.  Indeed, we have
	\begin{eqnarray}\label{eq:chull:implic}
		&& X(\partial_+ \Delta) \subset \text{Convex Hull }(\partial_- \Delta) \quad \text{implies}\\
	\nonumber &&	\qquad X(\Delta) \subset \text{Convex Hull }(\partial \Delta) = \text{Convex Hull }(\partial_- \Delta)
	\end{eqnarray}
	by the Classical convex hull theorem \cite{dhkw1}.

	When the free thread curvature $\kappa$ of $(X,\phi_-)$ is zero, the free thread is a straight segment.  We thus fulfill the condition of \eqref{eq:chull:implic} and our lemma follows.  Otherwise, we have by Lemma \ref{lem:firstvar:poskappa} that $\kappa > 0$.  Consider a point $p \in \partial_+^o \Delta$ parametrizing a thread point $X(p)$.  Let $F$ be an arbitrary linear function on $\R^3$.  By \eqref{eq:chull:implic}, proving our lemma reduces to showing
	\begin{equation}\label{eq:chull:ineq}F(X(p)) < \max_{\partial_- \Delta} F \circ X.
	\end{equation}
	We do this by showing that the harmonic function $h = F \circ X$ does not attain a local maximum at $p$.  To see this, extend $h$ across the boundary $\partial_+ \Delta$ near $p$.  It has an expansion whose lead term in a homogeneous harmonic polynomial $P(x,y)$.  If this polynomial is degree $2$ or higher, it is easy to find larger values of $h$ by moving into the interior of $\Delta$ from $p$.  If $P(x,y)$ is linear, then consider $(\nabla h)(p)$.  If this vector does not point normally out of $\Delta$, then we may move along a path in $\Delta$ from $p$ and find $h$ increasing to first order.  If this vector points normally out of $\Delta$, then we may use the fact that $\kappa > 0$ to show that as we move along $\partial \Delta$ away from $p$, we have $h$ increasing to second order.
\end{proof}

\subsection{Jointed pipe}

\begin{figure}
	\picholder{4in}{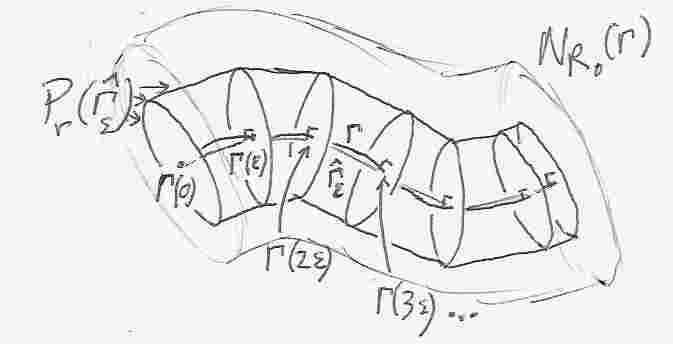}
	\mycaption{Jointed pipe.}
\label{fig:app:pipe}
\end{figure}

\begin{definition} To have a property in a \emph{finitely piecewise} manner will mean to have the property in finitely many pieces.  
\end{definition}
It will be technically useful to work with approximations of $\Gamma$ and its neighborhoods which have finitely piecewise properties.  Notation: a function $f(\e)$ is $o(\e)$ if $\limsup_{\e \to 0^+} \frac{|f(\e)|}{\e} = 0.$

\begin{lemma} \label{lem:pipe} Let $\Gamma$ be a $C^2$ curve with a simple $R_0$-normal neighborhood.  Then for small $\e>0$ we consider the finitely piecewise linear curve $\Ghat_\e$ with vertices $\Gamma(k\e), 0 \le k\e < \ell(\Gamma)$ and $\Gamma(\ell(\Gamma))$.  For $r < R_0 - o(\e)$, the set $P_r(\Ghat_\e)$ of points in $N_{R_0}(\Gamma)$ distance $r$ from $\Ghat_\e$ consists of a sequence of cylindrical surfaces joined along circles and portions of spheres, as shown in \figref{fig:app:pipe}.  We call this set a \emph{jointed pipe surface} of $\Gamma$.  There is an $\tilde{\e} = O(\e^2)>0$ so $P_r(\Ghat_\e)$ strictly encloses $N_{r-\tilde{\e}}(\Gamma)$ provided $r>\tilde{\e}$.
\end{lemma}
\begin{proof}
	Let us notate the vertices of $\Ghat_\e$ as $\Gamma(s_0),\Gamma(s_1),\ldots \Gamma(s_m)$.  In $N_{R_0} \Gamma$, the normal discs $D_{R_0}(\Gamma,s_k)$ are disjoint.  An application of Taylor's theorem shows that the radius $\rho$ disc $E_k$ bisecting the angles of $\Ghat_\e$ at $s=s_k$ has unit normal vector within $\kmax \e^2 + o(\e^2)$ of $D_\rho(\Gamma,s_k$.  Moreover, if we choose $\rho = R_0 - o(\e)$, the discs $E_k$ will be disjoint.  For $\rho_2 = \rho - o(\e)$, the $E_k$ bound $m-1$ compartments in $N_{\rho_2}(\Gamma)$.  In each compartment $C_k$ lies exactly one line segment of $\Ghat_\e$.  It is then not hard to see that for $r< \rho_2$, the portion of $P_r(\Ghat_\e)$ in the compartment is a portion of the radius $r$ cylinder about $C_k \cap \im \Ghat_\e$ and possibly portions of the radius $r$ spheres about $\partial(C_k \cap \im \Ghat_\e)$.  If the first and last of the $D_k$ are not parallel to the corresponding normal discs of $\Gamma$, then there may be compartments at the endponts; in these compartments the level sets of distance to $\Ghat_\e$ are just portions of spheres.

	Let $\Ghat_\e$ be parameterized on the two segments abutting an interior vertex $\Gamma(s_k)$ so in the Frenet coordinates of $\Gamma$ at $s_k$, $\Ghat_\e(s) = (s-s_k,y(s),z(s)).$  Then by Taylor's theorem,
	$$|\Ghat_\e(s)-\Gamma(s)| \le \kappa_\Gamma(0)\left(\frac{s-s_k}{\e} \half \e^2 - \half (s-s_k)^2\right) + o(\e^2)$$
	Maximizing the lead term at $s= s_k \pm \e/2$ shows that $d(\im \Gamma,\im \Ghat_\e) \le \frac18 \kmax \e^2 + o(\e^2) := \tilde{\e}$.  The last claim of the lemma follows by the triangle inequality.
\end{proof}

\subsection{Generic wire}\label{sec:generic}
We say a wire $\Gamma$ is generic if:
\begin{enumerate}
	\item The wire $\Gamma$ is $C^4$.
\item The curvature $\kappa_\Gamma$ does not vanish.  Hence $\Gamma$ is a Frenet curve --- it has an orthogonal frame consisting of $\Gamma_s, \Gamma_{ss}$ and the binormal $\eta_\Gamma(s)$. 
	\item The curvature $\kappa_\Gamma$ is a Morse function.  In other words, whereever its first derivative vanishes, its second derivative does not.
	\item The torsion $T_\Gamma$ crosses zero transversely.  In other words, wherever it vanishes, its first derivative does not.
	\item The torsion $T_\Gamma$ does not vanish at any critical point of curvature $\kappa_\Gamma$.
\end{enumerate}


\begin {thebibliography}{AA} 
\begin{dummy}\label{sec:biblio}\end{dummy}

\bibitem{alt} Alt, H.W.  Die Existenz einer Minimalfl\"ache mit freiem Rand vorgeschriebener L\"ange. Arch. Ration. Mech. Anal. \textbf{51}, 304-320 (1973).

\bibitem{surfstab} Barbosa, J.L.; do Carmo, M. On the size of a stable minimal surface in $\R^3$. Am. J. Math. \textbf{98}, no 2., 515--528 (1976).

\bibitem{douglas} Douglas, J. Solution of the problem of Plateau. Trans. Am. Math. Soc. \textbf{36}, 363--321 (1931).

\bibitem{dhkw1} Dierkes, U., S. Hildebrandt, A. K\"uster, O. Wohlrab.  Minimal Surfaces I: Boundary value problems.  Grundlehren der mathematischen Wissenschaften 295.  Springer-Verlag: Berlin, 1992.

\bibitem{dhkw2} Dierkes, U., S. Hildebrandt, A. K\"uster, O. Wohlrab.  Minimal Surfaces II: Boundary regularity.  Grundlehren der mathematischen Wissenschaften 295.  Springer-Verlag: Berlin, 1992.

\bibitem{dhl} Dierkes, U., Hildebrandt, S., Lewy, H. On the analyticity of minimal surfaces at movable boundaries of prescribed length. J. Reine Angew. Math. \textbf{379}, 100-114 (1987).

\bibitem{dziuk3} Dziuk, G. On the boundary behavior of partially free minimal surfaces. Manuscr. Math. \textbf{35}, 105--123 (1981).

\bibitem{ecker} Ecker, K. Area-minimizing integral currents with movable boundary parts of prescribed mass. Analyse non lin\'eaire. Ann. Inst. H. Poincar\'e \textbf{6}, 261-293 (1989).

\bibitem{federer} Federer, H.  Geometric Measure Theory.  Die Grundlehren der mathematischen Wissenschaften.  Band 153.  Springer-Verlag: New York, 1969.

\bibitem{morrey1} Morrey, C.B. On the solutions of quasi-linear elliptic partial differential equations. Trans. Am. Math. Soc. \textbf{43}, 126--166 (1938).

\bibitem{morrey3} Morrey, C.B. The problem of Plateau on a Riemannian manifold. Ann. Math. (2) \textbf{49}, 807--851 (1948).

\bibitem{nietreg1} Nitsche, J.C.C. Minimal surfaces with movable boundaries.  Bull. Am. Math. Soc. \textbf{77}, 746-751 (1971).

\bibitem{nietreg2} Nitsche, J.C.C.  The regularity of minimal surfaces on the movable parts of their boundaries. Indiana Univ. Math. J. \textbf{21}, 505-513 (1971).

\bibitem{nietreg3} Nitsche, J.C.C.  On the boundary regularity of surfaces of least area in Euclidean space. Continuum Mechanics and Related Problems in Analysis, Moscow, 375-377 (1972).

\bibitem{pilz} Pilz, R.  On the thread problem for minimal surfaces.  Calc. Var. \textbf{5} , 117-136 (1997).

\bibitem{rado} Rad\'o, T. Contributions to the theory of minimal surfaces. Acta Sci. math. Univ. Szeged \textbf{6}, 1-20 (1932).

\bibitem{myurl} Stephens, B.K. www.bkstephens.net.

\end{thebibliography}
\end{document}